\documentclass{amsart}
\usepackage{amsmath,amsthm,bbm}
\usepackage{amsfonts,amssymb}

\usepackage{enumerate}

\theoremstyle{plain} 

\newtheorem{thm}{Theorem}[section]
\newtheorem{cor}[thm]{Corollary}
\newtheorem{lem}[thm]{Lemma}
\newtheorem{prop}[thm]{Proposition}
\newtheorem{exam}[thm]{Example}

\newtheorem{defn}[thm]{Definition}

\newcommand{\thmref}[1]{Theorem~\ref{#1}}

\newcommand{\lemref}[1]{Lemma~\ref{#1}}

\newcommand{\propref}[1]{Propsition~\ref{#1}}

\theoremstyle{remark}
\newtheorem{rem}[thm]{Remark}

\numberwithin{equation}{section}

\def\df{\displaystyle \frac}

\def\la{{\langle}}
\def\ra{{\rangle}}
\def\f{\frac}

\def\vi{\varphi}
\def\({\left(}
\def \){ \right)}
\def\Bl{\Bigl}
\def\Br{\Bigr}

 \def\tr{{\triangle}}
 
\def\ta{\theta}
\def\al{{\alpha}}
\def\be{{\beta}}
\def\da{{\delta}}
\def\sa{{\sigma}}
 \def\a{{\alpha}}
 \def\b{{\beta}}
 \def\g{{\gamma}}
  \def\ga{{\gamma}}
 \def\k{{\kappa}}
 \def\t{{\theta}}
 \def\l{{\lambda}}
 \def\d{{\delta}}
 \def\o{{\omega}}
 \def\Og{{\Omega}}
 \def\s{{\sigma}}
 \def\va{\varepsilon}

 \def\BB{{\mathbb B}}

 \def\NN{{\mathbb N}}

 \def\RR{{\mathbb R}}
 \def\SS{{\mathbb S}}
 \def\TT{{\mathbb T}}

 \def\ZZ{{\mathbb Z}}

    \def\osc{\operatorname{osc}}

  \def\sph{\mathbb{S}^{d-1}}

\def\conv{\text{conv}}

 \def\dist{\mathtt{d}}

\newcommand{\wt}{\widetilde}
\newcommand{\wh}{\widehat}
\def\sub{\substack}

\DeclareMathOperator{\Arc}{Arc}

\def\conv{\text{conv}}
\def\diam{\text{diam}}
\def\p{\partial}
\def\ld{\lambda}
\def\bl{\bigl}
\def\br{\bigr}
\def\og{\omega}
\def\Ld{\Lambda}

\begin{document}

\title[]{
  Chebyshev-type cubature formulas for doubling weights on  spheres, balls and simplexes}

\author{Feng Dai}
\address{Department of Mathematical and Statistical Sciences\\
University of Alberta\\ Edmonton, Alberta T6G 2G1, Canada.}
\email{fdai@ualberta.ca}

\author{Han Feng}
\address{Department of Mathematics,
University of Oregon, Eugene OR 97403-1222, USA.}
\email{hfeng3@uoregon.edu}

\thanks{This work  was  supported
 by  NSERC  Canada under
grant  RGPIN 04702.  It was  conducted when the second author was a Ph.D student at the University of Alberta.}

\date{July 12, 2017}
\begin{abstract}


   This paper shows  that  given a doubling weight $w$ on  the unit sphere ${\mathbb{S}^{d-1}}$  of $\mathbb{R}^d$
   there exists a positive constant $K_{w,d}$
    such that  for each  positive integer  $n$ and
         each integer $N\ge   \max_{x\in {\mathbb{S}^{d-1}}} \frac {K_{w,d}} {w(B(x, n^{-1}))}$,  there exists a set of   $N$ distinct nodes $z_1,\cdots, z_N$  on ${\mathbb{S}^{d-1}}$  for which
       $$ \frac 1{w({\mathbb{S}^{d-1}})} \int_{{\mathbb{S}^{d-1}}} f(x) w(x)\, d\sigma_d(x)=\frac 1N \sum_{j=1}^N f(z_j),\   \  \forall f\in\Pi_n^d, \leqno{(\ast)} $$
    where $d\sigma_d$,  $B(x,r)$ and $\Pi_n^d$  denote  the  surface Lebesgue measure  on ${\mathbb{S}^{d-1}}$,  the spherical cap with center $x\in\sph$ and radius $r>0$, and the space of all spherical polynomials of degree at most $n$ on ${\mathbb{S}^{d-1}}$, respectively, and  $w(E)=\int_E w(x) \, d\sigma_d(x)$ for $E\subset {\mathbb{S}^{d-1}}$.
 If, in addition,  $w\in L^\infty({\mathbb{S}^{d-1}})$, then the above  set of  nodes can be chosen to be well separated:
$$\min_{1\leq i\neq j\leq N}\arccos(z_i\cdot z_j)\geq c_{w,d} N^{-\frac 1{d-1}}>0.$$
 It is further proved that the minimal number of nodes    $\mathcal{N}_{n} (wd\s_d)$  required in   ($\ast$)
  for a doubling weight $w$ on ${\mathbb{S}^{d-1}}$ satisfies
  $$ \mathcal{N}_n (wd\sigma_d) \sim \max_{x\in {\mathbb{S}^{d-1}}} \frac 1 {w(B(x, n^{-1}))},\   \  n=1,2,\cdots.$$
Proofs of these results rely on  new convex partitions of ${\mathbb{S}^{d-1}}$ that are regular with respect to a given weight $w$ and  integer $N$. Similar results are also established  on the unit ball and the standard simplex of $\mathbb{R}^d$.

   Our results  extend  the  recent  results of Bondarenko, Radchenko, and Viazovska  on spherical designs.

\end{abstract}
  \maketitle

\section{Introduction}
  A Chebyshev-type cubature  formula (CF)  of degree $n$ for a positive Borel   measure $\mu$ on a compact subset   $\Og$ of $\RR^d$ consists of finitely many nodes $z_1, \cdots, z_N$ in $\Og$ (which may repeat)  such that the  numerical integration formula
 \begin{equation}\label{1-1}
  \f 1{\mu(\Og)} \int_{\Og} P(x)\, d\mu(x)=\f 1N \sum_{j=1}^N P(z_j)
 \end{equation}
  holds for every  algebraic  polynomial $P$ of total degree at most $n$ in $d$ variables, where   $N$ is called the size of the formula.
       A  Chebyshev-type CF is called strict if  all the   nodes $z_j$ are distinct.
    We denote by   $ \mathcal{WN}_{n,\Og}(d\mu)$  (resp.,   $\mathcal{N}_{n,\Og}(d\mu)$)
  the minimal size  of the   Chebyshev-type CF  (resp.,  the  strict Chebyshev-type CF) of degree $n$
  for the  measure $\mu$ on $\Og$. We will drop the  subscript $\Og$ here   whenever the underlying domain $\Og$ is easily understood from the context and no confusion is possible.

    A central question concerning the Chebyshev-type CF  is
     to find  sharp asymptotic estimates of  the quantity  $\mathcal{N}_{n,\Og}(d\mu)$ or $\mathcal{WN}_{n,\Og}(d\mu)$  as $n\to\infty$. Study of this question starts with
    the classical work  of Bernstein \cite{Bern1, Bern2} who shows    that   $\mathcal{WN}_n(dx)\sim n^2$ for the Lebesgue measure $dx$ on the interval $[-1,1]$.  Here and throughout the paper,  the notation $a_n\sim b_n$ means that $c_1 a_n \leq b_n\leq c_2 a_n$ for some positive constants $c_1, c_2$   independent of $n$ (called constants of equivalence).

    Bernstein's methods  have been extended and  developed   in a series of  papers  of    Kuijlaars  (see, for instance,  \cite{Kui93, Kui95-1,Kui95-2,Kui95-3,Kui95-4}),  who,
 in particular,   proves   that  for   the Jacobi measures $d\mu_{\al,\be}(t)=(1-t)^\al(1+t)^\be \, dt$ on  $[-1,1]$  with nonnegative parameters  $\al, \be\ge 0$,
 \begin{equation}\label{1-2-intr}
    \mathcal{N}_n( d\mu_{\a,\b})\sim n^{2+2\max\{\a,\b\}},\   \   \  n=1,2,\cdots.
 \end{equation}
 Kuijlaars \cite{Kui93, Kui95-1} also noticed that his techinique in general does not work for the Jacobi measures $d\mu_{\al,\be}(t)$ with negative  parameters $\al,\be>-1$, although he    was able to    prove  a stronger result in \cite{Kui95-5} implying  the estimate  \eqref{1-2-intr}  for  the case of  $\al=-\f12$,    $-\f12-\ld_0\leq \be <-\f12$ and   some positive constant $\ld_0$.   The  estimate  \eqref{1-2-intr} was proved  recently  by  Kane\cite{Kane}  for   $\a,\b\geq -\f 12 $, and by  Gilboa and Peled \cite{GilPel} for the general case of  $\a,\b>-1$. The very interesting  work  of  Gilboa and Peled \cite{GilPel} also  uses the method of Kane\cite{Kane}  to establish sharp bounds on the size of the Chebyshev-type CFs in one dimension for all doubling weights  with  some excellent discussions on non-doubling case.
  For more information on  the Chebyshev-type CFs in one variable,    we refer to  \cite{Fo1, Fo2, Gaut, Kor, Peled} and references within.

For the Chebyshev-type CFs in several variables, the most well studied case is that of spherical designs, introduced by   Delsarte, Goethals, and Seidel \cite{DGS} in 1977.
Let $\sph:=\{x\in\RR^d:\  \ \|x\|=1\}$  denote the  unit sphere of  $\RR^d$ equipped with the usual surface Lebesgue measure $d\s_d(x)$, where $\|\cdot\|$ denotes the Euclidean norm.
A spherical $n$-design  on $\sph$  is a
Chebyshev-type CF of degree $n$ for the  measure $d\s_d(x)$ on $\sph$. It is not difficult to show a spherical $n$-design on $\sph$ must have size at least $c_d n^{d-1}$; that is, $\mathcal{N}_n(d\s_{d})\ge c_d n^{d-1}$,
(see, for instance, \cite{DGS}).  It was conjectured by Korevaar and Meyers \cite{KoMe}  that
$\mathcal{N}_n(d\s_{d})\leq  c_d n^{d-1}$.
Much work had been done towards this conjecture (see \cite{Kane} and the references thererin). This conjecture   was recently confirmed  in the breakthrough work of Bondarenko, Radchenko, and Viazovska \cite{Bond1, Bond2},
which shows  that there exist  positive constants $K_d$ and $c_d$ depending only on $d$  such that given  every integer $N>K_d n^{d-1}$ there exists a  spherical $n$-design  consisting of $N$ distinct nodes  $z_1, \cdots, z_N\in\sph$ with
$\min_{1\leq i\neq j\leq N}\dist(z_i, z_j)\ge c_d N^{-\f 1{d-1}}$.


In  the above mentioned  interesting  paper  \cite{Kane},
  Kane develops  techniques  different from those of \cite{Bond1, Bond2} to
  establish bounds on the size of Chebyshev-type  CFs on   rather general path-connected topological spaces.
  In particular, his techniques can be applied to prove the existence of spherical $n$-designs on  $\sph$ of size $O_d( n^{d-1} (\log n)^{d-2})$, which is   only slightly worse  than the asymptotically optimal estimate  $O_d(n^{d-1})$.
  In the very general setting of path-connected topological spaces, the work of Kane \cite{Kane}  proves the existence of the Chebyshev-type  CFs
  whose size is
roughly the square of   the optimal size   conjectured in \cite{Kane}.
 Using the method of \cite{Bond1}, Etayo,  Marzo and Ortega-Cerd\`{a} \cite{Eta} establish  asymptotically optimal bounds on the size of the Chebyshev-type  CFs
 on compact algebraic manifolds,  confirming  Kane's  conjecture in certain sense for this specific setting.

  Some earlier related works  regarding the  Chebyshev-type  CFs  on  certain multi-variate domains can be found in the papers of  Kuperberg \cite{Ku1, Ku2, Ku3}.
     One can also find some interesting results   on Chebyshev-type CFs on discrete spaces such as  combinatorial designs and Hadamard matrices,  in  \cite{KLP} and the references therein.

 One of the main purposes in this paper is to  extend the methods of Bondarenko, Radchenko, and Viazovska \cite{Bond1, Bond2} to determine the asymptotically optimal bounds on the size of  the Chebyshev-type CFs  for   doubling weights   on  the unit sphere $\sph$ and other related  domains. Our results are mainly  for  the case of more variables, whereas  the results in one variable were mostly established  in the recent paper \cite{GilPel}.

 Let us start with  some necessary notation.
   Denote by $\Pi_n^d$  the space  of all real algebraic polynomials   of total  degree at most $n$ in $d$ variables:
 $$\Pi_n^d:=\text{span}\Bl\{x^\al=x_1^{\al_1}\cdots x_d^{\al_d}:\   \   \al=(\al_1,\cdots, \al_d)\in\NN_0^d, \  \ |\al|:=\sum_{j=1}^d \al_j\leq n\Br\},$$  where $\NN_0$ denotes the set of all nonnegative integers.
Let $\dist(x,y)$ denote the geodesic distance on $\sph$; that is, $\dist(x, y):=\arccos (x\cdot y)$ for $  x, y\in\sph$.
  Denote by $B(x,r)$ the spherical cap $\{ y\in\sph:\  \ \dist(x,y)\leq r\}$ with center $x\in\sph$ and radius $r>0$.
 A weight function $w$ on $\sph$ (i.e., a nonnegative integrable function on $\sph$ )  is said to satisfy the doubling condition if there exists a positive constant $L$ such that
$$w(B(x, 2r)) \leq L w(B(x,r))\   \ \forall x\in\sph,\   \  \forall r>0,$$ where we write  $w(E):=\int_E w(x)\, d\s_{d}(x)$ for $E\subset \sph$,  and
the least constant $L=L_w$ is called the doubling constant of $w$.   Given a doubling weight $w$ on $\sph$, there exists a constant $d-1\leq s_w \leq \f {\log L_w}{\log 2}$ such that
\begin{equation}\label{1-5-2017}
  w(B(x, \ld t))\leq 2 \ld^{s_w} w(B(x,t)),\    \  \forall \ld>1,\   \  \forall t>0,\    \   \forall x\in\sph.
\end{equation}

Many of  the weights
that appear in analysis on $\sph$  satisfy the doubling condition;  in
particular,  all weights of the form
\begin{equation}\label{1-3} w_\al (x) = \prod_{j=1}^d|x_j|^{\al_j},\     \ x\in\sph,\ \  \al=( \al_1,\cdots, \al_d)\in (-1, \infty)^d.\end{equation}

   In this paper, we will prove the following weighted extension of  the result of Bondarenko, Radchenko, and Viazovska \cite{Bond1, Bond2} on spherical designs:

\begin{thm}\label{sphdesign201}
  Let $w$ be a  doubling weight on $\sph$  normalized by $w(\sph)=1$.  Then there exists a positive constant $K_w$ depending only on the doubling constant of $w$ and the dimension $d$ such that  for each given positive integer  $n$, and
        every integer $N\ge  \max_{x\in \sph}\f {K_w} {w(B(x, n^{-1}))}$,  there exist  $N$ distinct nodes $z_1,\cdots, z_N\in\sph$ for which
     \begin{equation}\label{1-5:2017}\int_{\sph} P(x) w(x)\, d\s_d(x) =\f 1 N \sum_{j=1}^N P(z_j),\   \  \forall P\in \Pi_n^d.\end{equation}
     If, in addition, $w\in L^\infty(\sph)$,  then the  set
         of nodes $\{z_1,\cdots, z_N\}$ above   can be chosen to be well separated:
$$\min_{1\leq i\neq j\leq N}\dist(z_i,z_j)\geq c_\ast N^{-\f 1{d-1}},$$
where  $c_{\ast}$ is a positive constant depending only on $\|w\|_\infty$ and the doubling constant of $w$.
  \end{thm}

As a corollary of Theorem \ref{sphdesign201}, we have
\begin{cor}\label{cor-1-2} Given  a  doubling weight $w$ on $\sph$,   the minimal sizes   of   the  strict    Chebyshev-type CFs
  for the  measure $w(x)d\s_d(x)$ on $\sph$ satisfy
  $$ \mathcal{N}_n (wd\s_d) \sim \max_{x\in \sph}\f {w(\sph)} {w(B(x, n^{-1}))},\ \ n=1,2,\cdots$$
  with the constants of equivalence depending only on $d$ and  the doubling constant of $w$.
\end{cor}

Several remarks are in order.
\begin{rem}\  \
\begin{enumerate}[\rm (i)] \item    By \eqref{1-5-2017}, it is easily seen that
$\max_{x\in \sph}\f 1 {w(B(x, n^{-1}))}\leq 2\pi^{s_w} n^{s_w}$. Thus, Corollary \ref{cor-1-2} particularly  implies that
$ \mathcal{N}_n (wd\s_d) \leq C n^{s_w}<\infty$ for every doubling weight $w$ on $\sph$.

  \item For the weights $w_\al$ given in \eqref{1-3}, a straightforward calculation shows that
$$ w_\al (B(x, \ta)) \sim \ta^{d-1} \prod_{j=1}^d (|x_j|+\ta)^{\al_j},\    \   x\in\sph,\  \  \ta\in (0,\pi].$$
By  Corollary \ref{cor-1-2}, this implies  that
$$ \mathcal{N}_n (w_\al d\s_d) \sim  n^{d-1+\sum_{j\in I_\al}  \al_j-\max\{\al_{\min},0\}},$$
where
$\al_{\min}=\min_{1\leq i\leq d}\al_i$ and $I_{\al}=\{ i:\  \ 1\leq i\leq d,\   \ \al_i\ge  0\}$.

\item Note that only  Chebyshev-type CFs with distinct nodes  are involved  in Theorem \ref{sphdesign201}. However,
   it is worthwhile to point out that a slight modification of our proof   shows that for each integer $N\ge  \max_{x\in \sph}\f {K_w} {w(B(x, n^{-1}))}$ there exists a Chebyshev-type CF of degree $n$ and size $N$ for the measure $w(x)d\s_{d}(x)$ consisting of a large number of multiple nodes (i.e., repeated notes), and  with the number $N_n$  of distinct nodes satisfying  $N_n\sim  n^{d-1}$.
\end{enumerate}

\end{rem}

 While the proof of Theorem \ref{sphdesign201}  follows  the methods of the papers  \cite{Bond1, Bond2},   it is more technical and involved than the corresponding unweighted case  due to the fact that  the measure $w(x) d\s_d(x)$ is not rotation-invariant,  which means that in general,    the weighted measure $w(B(x,r))$ of a spherical cap  not only depends on radius $r$ but also on the center $x$.

 An important ingredient used in our proof is
   the convex partition of  $\sph$ that is regular with respect to a given weight.
 Recall that a  subset $A\subset \sph$ is  geodesically convex if any two points $x, y\in A$  can be joined  by a geodesic arc that lies entirely in $A$, whereas  a finite collection $\{R_1, R_2, \cdots, R_N\}$ of closed geodesically convex subsets of $\sph$ is called
    a convex partition of $\sph$ if $\sph=\bigcup_{j=1}^N R_j$ and the interiors of the  sets $R_j$ are
     pairwise disjoint.
     Our result  on  regular convex partitions of the weighted sphere can be stated   as follows:

 \begin{thm} \label{thm-partition}Let $w$ be a normalized  weight on $\sph$  (i.e., $w(\sph)=1$) satisfying that  $w(B)>0$ for every nonempty spherical cap $B\subset \sph$. Assume that $r\in (0, \pi]$ and  $N$ is a positive integer satisfying   $\min_{x\in\sph} w(B(x, r))\ge \f 1N$.  Then
there exists a convex partition
 $\{R_1, \cdots, R_M\}$ of $\sph$    satisfying that
  $ B(x_j, c_d' r )\subset R_j \subset B(x_j, c_d r)$ and $Nw(R_j)\in\NN$ for every  $1\leq j\leq M$, where $c_d, c_d'$ are two positive constants depending only on $d$.
 \end{thm}

Several remarks on Theorem \ref{thm-partition} are in order:

\begin{rem}
\begin{enumerate}[(i)]
\item Theorem \ref{thm-partition} seems to be of independent interest. It may   have  other  applications in discrepancy theory and optimization of discrete energies on the sphere (see, for instance, \cite{KS, KuiSaff}).

\item In the case of $w\equiv 1$,   Theorem \ref{thm-partition} with $M=N$ and  $\s_d(R_j)=\f 1N$ for $j=1, \cdots, N$  is due to Bondarenko, Radchenko, and Viazovska
 \cite[Proposition 1]{Bond2}.
 The proof of Theorem \ref{thm-partition} is, however,  much more involved than the corresponding case of the surface Lebesgue measure $d\s_d$ due to the fact that the measure $w(x)d\s_d(x)$ is not rotation-invariant.

 \item Since   the   weight $w$  may have zeros or discontinuities on $\sph$,  the integers $k_j:=Nw(R_j)$ in Theorem \ref{thm-partition} in general  depend on the location of $R_j$, which is different from the corresponding  unweighted case.  This difference also increases the technical difficulties  of  the proof of Theorem \ref{sphdesign201}.
 \item If, in addition, the weight    $w$  satisfies the doubling condition, then according to \eqref{1-5-2017}, there exists a constant $c_w$ such that the condition,   $\min_{x\in\sph} w(B(x, r))\ge \f 1N$,  is satisfied whenever $r\ge c_w N^{-1/s_w}$.
\end{enumerate}
\end{rem}

  If we drop the  convexity requirement in the partition in   Theorem \ref{thm-partition}, we may deduce   the following corollary:

 \begin{cor}\label{cor-1-4}Under the conditions of Theorem \ref{thm-partition},
there exists a partition  $\{R_1, \cdots, R_N\}$ of $\sph$  such that
 $w(R_j)=\f 1N$ and $\diam (R_j)\leq C_w  r$ for $j=1,\cdots, N$.
\end{cor}

   Corollary \ref{cor-1-4} follows directly from   Theorem \ref{thm-partition} due to  the fact that if $E\subset \sph$ and   $0<\al<w(E)$, then there exists a subset $F$ of $ E$ such that $w(F)=\al$.

%
%
%
%
%

In order to establish similar results on other domains, we also  need to  consider weights on $\sph$  that are symmetric under certain reflection groups, in which case it can be shown that   the set of nodes in the corresponding  Chebyshev cubature formula enjoys    the same symmetry. Let us first describe briefly   some necessary notation.
Given $j=1, 2,\cdots, d$, we denote by $\tau_j$  the reflection with respect to the coordinate plane $x_j=0$; that is,
$$x \tau_j =(x_1, \cdots, x_{j-1}, -x_j, x_{j+1},\cdots, x_d),\   \    x\in\RR^d.$$
Denote by  $\ZZ_2^d$  the abelian
reflection group  generated by the reflections $\tau_1,\cdots, \tau_d$.
A weight $w$ on $\sph$ is called $\tau_j$-invariant for a given $j$ if $w(x\tau_j) =w(x)$ for all $x\in\sph$, and  is called $\ZZ_2^d$-invariant if
 it is $\tau_j$-invariant for every $j=1,2, \cdots, d$.
Similarly, we say  a finite subset $\Ld$ of $\sph$  is    $\tau_j$-invariant for a given $j\in\{1,2,\cdots,d\}$ if $\Ld=\{ x\tau_j:\  \ x\in\Ld\}$, whereas it  is  $\ZZ_2^d$-invariant if it  is $\tau_j$-invariant for every $j=1,2, \cdots, d$.
For simplicity, we set $\SS_{\text{inter}}^{d-1}:=\{ x\in \sph:\  \ x_i\neq 0,\  \  i=1,2,\cdots, d\}$.

A slight modification of the proof of Theorem \ref{sphdesign201}  yields the following result.
\begin{cor}\label{cor-1-3:2017} Let $w$ be a $\ZZ_2^d$-invariant (resp. $\tau_d$-invariant)  doubling weight on $\sph$  normalized by $w(\sph)=1$.  Then  the conclusions of Theorem ~\ref{sphdesign201}  hold with  the   set of nodes $\Ld:=\{z_1, \cdots,z_N\}$ being  $\ZZ_2^d$-invariant (resp. $\tau_d$-invariant)   and     contained in the set $\SS_{\text{inter}}^{d-1}$,   (i.e. none of the nodes lies in the coordinate planes).
\end{cor}

 The weighted results on the sphere also allow us to establish similar results on the  unit ball $\BB^d:=\{x\in\RR^d:\  \ \|x\|\leq 1\}$ and on the simplex
 $\TT^d:=\Bl\{ x=(x_1,\cdots, x_d)\in \RR^d:\  \ x_1, \cdots, x_d\ge 0,  |x|:=\sum_{j=1}^d |x_j|\leq 1\Br\}.$  Indeed, it was   observed  by Y. Xu \cite{Xu7} that   CFs  on $\BB^d$  and    $\TT^d$ are closely related to weighted CFs  on $\SS^d$. To describe this connection,  we use the notation $\Og$ to denote either the unit ball $\BB^d$ or the simplex  $\TT^d$ equipped with the usual Lebesgue measure $dx$. Consider   the following  metric $\rho_{\Og}$ on the domain $\Og$:  for $x, y\in \Og$,
 \begin{equation}\label{metric-ball}
\rho_{\Og} (x,y):=\begin{cases} \|x-y\|+\bl|\sqrt{x_{d+1}}-\sqrt{y_{d+1}}
\br|,&\  \ \text{if $\Og=\BB^d$},\\
\arccos ( \sum_{j=1}^{d+1}\sqrt{x_jy_j}), &\  \ \text{if $\Og=\TT^d$},\end{cases}
\end{equation}
 where  we write
 $$x_{d+1}=\left\{
             \begin{array}{ll}
               \sqrt{1-\|x\|^2}, & \hbox{$x\in \BB^d$;} \\
               1-|x|=1-x_1-\cdots-x_d, & \hbox{$x\in \TT^d$.}
             \end{array}
           \right.
 $$ Denote by $B_\Og(x,r)$ the ball $\{y\in\Og:\  \ \rho_\Og (x,y)\leq r\}$ with center $x\in\Og$ and radius $r>0$.
A doubling weight  on $\Og$ is  a weight $w$ on $\Og$ satisfying that
$$ w(B_\Og (x, 2r))\leq L w(B_{\Og}(x,r)),\  \ \forall x\in\Og,\  \ \forall r>0$$
for some positive constant $L$, where $w(E)=\int_E w(x)\, dx$ for $E\subset \Og$.

 Next, we define the mapping $\phi_\Og :\  \ \SS^d\to \Og$ by
 $$\phi_{\Og}(x)=\begin{cases} (x_1, x_2, \cdots, x_d),&\   \ \text{if $ \Og=\BB^d$},\\
 (x_1^2, x_2^2, \cdots, x_d^2), &\   \  \text{if $ \Og=\TT^d$},
 \end{cases}\   \  \  x=(x_1, x_2, \cdots, x_d, x_{d+1})\in \SS^d.$$
 A   change of variables  shows that for each  integrable function $f$ on $\Og$ (see \cite{Xu7})
 \begin{equation}\label{1-7:2017}
   \int_{\Og} f(x)\, dx= c_d \int_{\SS^d} f( \phi_\Og(x)) \Phi_{\Og} (x)\, d\s_{d+1}(x),
 \end{equation}
 where $\Phi_{\BB^d}(x)=|x_{d+1}|$ and $\Phi_{\TT^d} (x)=\prod_{j=1}^{d+1} |x_j|$ for $x=(x_1,\cdots, x_{d+1})\in\SS^d$.
Note that if $f\in \Pi_n^d$, then $f\circ \phi_{\Og}$ is a $\tau_{d+1}$-invariant or $\ZZ_2^d$ invariant spherical polynomial on $\SS^{d+1}$ of degree at most $n$ or $2n$ depending on whether $\Og=\BB^d$ or $\Og=\TT^d$.

 Given a weight $w$ on $\Og$, we define a weight $ w_\Og$ on the sphere $\SS^d$ by
 \begin{equation}\label{}
    w_\Og(x, x_{d+1})=\begin{cases}w (x)|x_{d+1}|,&\  \  \text{if $\Og=\BB^d$},\\
    (\prod_{j=1}^{d+1} |x_j|) w (x_1^2,\cdots, x_d^2),&\  \  \text{if $\Og=\TT^d$},
    \end{cases}
 \end{equation}
 where  $x\in\BB^d$ and $(x, x_{d+1})\in\SS^d$.
Clearly  $w_{\Og}$ is a $\tau_{d+1}$ invariant or $\ZZ_2^d$ invariant weight on $\SS^d$ satisfying $w_{\Og}(z)=w(\phi_\Og(z))\Phi_\Og (z)$ for $z\in\SS^d$. Thus,  according to  \eqref{1-7:2017},    $w$ is a doubling weight on $\Og$ if and only if $ w_\Og$ is a doubling weight on the sphere $\SS^d$, and furthermore, $w(E)=w_\Og (E^{\Og})$ for every measurable  $E\subset \Og$, where
%
%
  $$
 E^{\Og}=\begin{cases}
 \{ (x, \sqrt{1-\|x\|^2}):\  \  x\in E\},&\   \    \text{ if $\Og=\BB^d$},\\
  \{ (\sqrt{x_1},\cdots, \sqrt{x_d}, \sqrt{1-|x|}):\  \  x\in E\},   &\  \  \text{ if $\Og=\TT^d$}.\end{cases}$$

 The following are some  examples of doubling weights on the ball $\BB^d$  or the simplex $\TT^d$:
 \begin{exam} \begin{enumerate}[\rm (i)] \item  $
    w(x) :=\prod_{j=1}^m \|x-\eta_j\|^{\al_j}$, $ x\in \BB^d$, where    $\al_1,\cdots, \al_m>-1$ and  $\eta_1,\cdots, \eta_m$ are distinct points in $\BB^d$.
    \item $w(x)=(1-\|x\|)^{\al_{d+1}}\prod_{j=1}^d |x_j|^{\al_j}$, $ x\in \BB^d$, where    $\al_1,\cdots, \al_{d+1}>-1$.
        \item $w(x)=(1-|x|)^\mu\prod_{j=1}^d |x_j|^{\k_j}$, $x\in \TT^d$, where $\k_1,\cdots,\k_d, \mu>-1$.
 \end{enumerate}
\end{exam}

As a direct consequence of \eqref{1-7:2017} and  Corollary \ref{cor-1-3:2017}, we deduce
\begin{cor}\label{sphball}
Let $\Og$ denote either the unit  ball  $\BB^d$ or the simplex $\TT^d$, and
 $w$  a  doubling weight  on $\Og$ normalized by $w(\Og)=1$. Let
 $M^\Og_{n,w}:=\max_{x\in \Og}\f 1 {w(B_{\Og}(x,n^{-1}))}$ for $n=1,2,\cdots$.
 Then there exist positive  constants $K_{d,w}$, $C_{d, w}$, $c_{d,w}$ depending only on $d$ and the doubling constant of $w$ such that the
   following statements holds:
  \begin{enumerate}[\rm (i)]
  \item For $n\in\NN$, we have
  $ c_{d,w}M^\Og_{n,w}\leq \mathcal{N}_{n,\Og}( w(x)dx)\leq C_{d,w}M^\Og_{n,w}$.
 \item   For  each   integer $N\ge K_{d,w} M_{n,w}^{\Og}$,  there exists  a set of $N$ distinct nodes $z_1,\cdots, z_N$ in the interior of $\Og$ such that
     $$\int_{\Og} P(x) w(x)\, dx =\f 1 N \sum_{j=1}^N P(z_j),\   \  \forall P\in \Pi_n^d.$$
     If, in addition, $w_{\Og}\in L^\infty(\SS^d)$,  then the  set
         of nodes  $z_1, \cdots, z_N\in \Og$    can be chosen so that
$\min_{1\leq i\neq j\leq N}\rho_{\Og}(z_i,z_j)\geq c_{*} N^{-\f 1{d}}$, where $c_\ast$ is a positive constant depending only on $\|w_\Og\|_\infty$ and the doubling constant of $w$.\end{enumerate}

\end{cor}


The rest of the paper is organized as follows.  The first two sections are devoted to  the proof of Theorem ~ \ref{thm-partition}. To be more precise, in Section ~2 we prove  Theorem ~ \ref{thm-partition} under the additional condition $\f N {2^{d-2}}\in\NN$, which is technically easier, but already contains some crucial ideas.   The proof of Theorem \ref{thm-partition} for the general case of positive integer $N$ is more complicated and involved, and  is given in Section ~3.   A crucial ingredient used in the  proof in Section ~3 is   the   family of  nonlinear dilations $T_\al$ on the sphere that preserve geodesic simplexes.
 Section ~4 contains some preliminary lemmas that are either known or relatively easy to prove, but will be needed in the proof of Theorem \ref{sphdesign201}.   The final section,  Section 5,  is devoted to the proofs of Theorem \ref{sphdesign201}  and Corollary \ref{cor-1-2}. One of the main difficulties in our proofs  comes from the fact that the positive integers $N w(R_j)$ in Theorem \ref{thm-partition} ( i.e., the theorem on  convex partition) may not be equal to one, which is different from the unweighted case.

For the rest of  the paper,  we use the notation
$C_w$, $L_w$ , etc. ($c_w$ , $\ld_w $, etc.) for sufficiently large (small) constants depending only on the dimension $d$ and the doubling constant of $w$.
 Sor simplicity,  we will always assume that $d\ge 3$ (i.e., $\sph\neq \SS^1$).  The case of $d=2$ (i.e., $\sph=\SS^1$)  can be treated similarly and is,  in fact,  much simpler.


\section{Proof of Theorem \ref{thm-partition}:  Case 1.   }

 This section is devoted to the proof of Theorem \ref{thm-partition} under the additional assumption  $\f N {2^{d-2}}\in\NN$.

\subsection{Preliminaries}

 We start with the concept of geodesic simplex, which will play a crucial role in our proof.
\begin{defn} A  subset $S$ of $\sph$ is called a geodesic simplex of $\sph$ spanned  by a set of linearly independent vectors $\xi_1, \cdots, \xi_d \in\sph$ and is denoted by $S=\conv_{\sph}\{\xi_1,\cdots,\xi_{d}\}$ if
$$S=\Bl\{ g(x)=\f x{\|x\|}:\  \ x=\sum_{j=1}^d t_j \xi_j,\   \  t_1, \cdots, t_d \ge 0\  \  \text{and }\  \ \sum_{j=1}^d t_j=1\Br\}.$$
A  geodesic simplex $S$ is called admissible if,  in addition, the vectors $\xi_2,\cdots, \xi_{d}$ are mutually orthogonal, $\arccos \xi_1\cdot \xi_2  \in [\f \pi 6, \f \pi 2]$ and $\xi_1\cdot \xi_j=0$ for $j=3,\cdots, d$, in which case $\{\xi_1, \cdots, \xi_d\}$ is called an admissible subset of $\sph$.
\end{defn}

%


Next, we introduce some necessary notation.
For any two distinct points $\xi,\eta\in\sph$, we denote by $\Arc(\xi,\eta)$ the geodesic arc connecting $\xi$ and $\eta$; that is, $\Arc(\xi, \eta)=\{t_1\xi+t_2\eta:\  \ t_1, t_2\ge 0,\  \ t_1\xi+t_2\eta\in\sph\}$.
For $x\neq y\in \RR^d$, let $[x,y]$  denote the line segment $\{ (1-t)x+ty:\  \ t\in [0,1]\}$.
A partition $\{E_1,\cdots, E_n\}$ of a set $E\subset \RR^d$ is called convex if each $E_j$ is a convex subset of $\RR^d$.  Denote by $\conv_{\RR^d}(E)$ the regular  convex hull of a set $E$ in $\RR^d$.
Given a set $A=\{\eta_1,\cdots, \eta_{d}\}$ of $d$ linearly independent vectors in $\RR^d$, the convex hull $T=\conv_{\RR^d}(A)$  is called  a  surface  simplex spanned by the set $A$.

The following lemma can be verified by straightforward calculations.
\begin{lem}\label{prop-2-2} Let
$ T:=\conv_{\RR^d}\{\eta_1, \cdots, \eta_d\}$
denote a surface  simplex  in $\RR^d$ and $S=\conv_{\sph} \{ \eta_1, \cdots, \eta_d\}$
a geodesic simplex in $\sph$ both being spanned by a set of linearly independent vectors $\eta_1, \cdots, \eta_d\in\sph$. Let $H_T$ denote the hyperplane passing through the points $\eta_1,\cdots, \eta_d$.  Let $g(x)=\f x{\|x\|}$ for $x\in\RR^d\setminus\{0\}$.
     Then the following statements hold:

\begin{enumerate}[\rm (i)]
\item  $
   x\in  S$ if and only if $x\in \sph$ and $x=\sum_{j=1}^d t_j \eta_j$
  for some $t_1, \cdots, t_d\ge 0$.

 \item  $g: x\mapsto \f {x}{\|x\|}$ is  a bijective   continuously differentiable mapping from $T$ onto $S$ that maps each  convex subset of $T$ to a  geodesically convex subset of $S$. Furthermore, $g([x,y])=\Arc(g(x), g(y))$ for any two distinct points $x, y \in T$.
     \item  For  each nonnegative measurable function $F: S\to [0, \infty)$,
                $$\int_{S} F(x)\, d\s_{d}(x) = \int_{T} F(g(y)) \f {a} {\|y\|^d}\, dy,$$
                where $a:=\min_{x\in T}\|x\|$.

\item  If  $\min_{x\in T}\|x\|=a$, then
     $$ \f a{2(a+1)} \|x-y\|\leq \dist (g(x), g(y))\leq \f \pi {2a}\|x-y\|,\   \    \  \forall x, y\in H_T\cap 2\BB^d.$$

\end{enumerate}
\end{lem}

The following  geometric fact, which can also be easily verified through straightforward calculations, will be used frequently in our proof.
\begin{lem} \label{lem-6-2}Assume that $x=\xi\cos \t_1 +\eta_1 \sin \t_1$ and $y=\xi\cos \t_2 +\eta_2 \sin \t_2$, where $\xi, \eta_1, \eta_2\in \sph$, $\eta_i\cdot \xi=0$ and $0\leq \t_i \leq \pi$ for $i=1,2$. Then
\begin{equation}\label{}
     \sin^2\f {\dist(x,y)}2= \sin^2 \f {\t_1-\t_2}2 +(\sin \t_1 \sin \t_2)\sin^2\f {\dist(\eta_1, \eta_2)}2.
\end{equation}
Thus,
\begin{align*}
    |\t_1-\t_2|&\leq \dist (x,y),\   \   \dist(\eta_1,\eta_2) \leq \f \pi2\f {\dist(x,y)}{\sqrt{\sin \t_1\sin\t_2}}\\
    \dist(x,y)&\leq \f \pi 2 |\t_1-\t_2| +\f \pi 2 \sqrt{\sin\t_1\sin\t_2} \dist(\eta_1, \eta_2).
\end{align*}
\end{lem}

Finally, we need the following geometric property concerning convex sets in $\RR^d$.

\begin{lem}\label{lem-2-6}
    If $G$ is a  convex subset of  $\RR^d$ with  diameter $ 2 r$ and volume $\ge c_d  r^{d-1} a$ for some $a\in (0, r]$, then  $G$ contains a Euclidean ball of radius $c_d' a$ for some positive constant  $c_d'$.
   \end{lem}

   \begin{proof}
   We   use induction on the dimension $d$. The conclusion  holds trivially for $d=1$ since every convex subset of $\RR$ must be an interval.  Now assume that the conclusion has been  proven in   $\RR^{d-1}$, and we will deduce it for the case of  $\RR^d$ as follows.
   We denote by  $m_d$ the $d$-dimensional Lebesgue measure (i.e., the $d$-dimensional  Hausdorff measure).  Let $p, q\in \overline{G}$ be such that $2r=|p-q|=\text{diam}(G)$, and  let $\xi$ denote the unit vector in the direction of $q-p$. Then
$\overline{G}\subset \{ x\in\RR^d:\  \  0\leq (x-p)\cdot \xi\leq 2r\}.$
For $0\leq t\leq 2r$, let
$G( t)=\{ x\in \overline{G}:\  \ (x-p)\cdot \xi=t\}$.
 It is easily seen that each slice $G(t)$  is a compact convex subset of the hyperplane $S(t):=\{ x\in\RR^d:\  \ (x-p)\cdot \xi=t\}$,  and moreover,
\begin{align*}
    cr^{d-1} a \leq m_d  (G)=\int_{0}^{2r} m_{d-1} (G(t))\, dt.
\end{align*} Thus,  there must exist  $0<t_0<2r$ such that
$ m_{d-1} (G(t_0))>c'r^{d-2}a>0.$
Since $\diam(G(t_0))\leq \diam(G)=2r$, it follows  by the induction hypothesis   that
$G(t_0)$ contains a  $(d-1)$-dimensional ball  $B=\{ x\in S(t_0):\   \  |x-z_0|\leq ca\}$ for some  $z_0\in S(t_0)$ and $c=c_d>0$.
Without loss of generality,  we may assume that $r\leq t_0\leq 2r$  since otherwise we  interchange the order of the points $p$ and $q$.
It suffices to show that  the convex hull $H$ of the set $B\cup \{p\}$ contains a ball of radius $c_d a$.  Indeed, by rotation invariance,   we may assume  that
$p=0$ and $\xi=e_d$.
We then  write $z_0=(u_0, t_0)$ and let $Q$ denote the cube in $\RR^{d-1}$ centered at $u_0$ and having side length $\va_d a$ for a sufficiently  small constant $\va_d$.  Clearly, the rectangle  $R=Q\times [(1-\f{\va_d a}r)t_0, t_0]$ contains a $d$-dimensional  ball of radius  $c'_d a$. However,
a straightforward calculation shows that $R\subset H$.
  \end{proof}

\subsection{Organization of the proof}
We divide the proof into two main steps. At the first step, we  prove
\begin{prop}\label{prop-1-9-appendix} Let $w$ be a weight  on $\sph$ normalized by   $w(\sph)=1$. If $\f N {2^{d-2}}\in\NN$ and $N\ge
\max_{x\in\sph}\f 1{ w(B(x, 10^{-d}))}$,  then there exists a partition $\{T_1, \cdots, T_m\}$ of $\sph$   such  that each $T_j$ is an  admissible geodesic simplex in $\sph$ satisfying $N w(T_j)\in\NN$.
\end{prop}

To state our  main result in the next step, we need to introduce some notation.   Given a surface  simplex $T$ in $\RR^d$, we  denote  by  $H_T$ the hyperplane in which   the surface simplex $T$ lies, and  $B(x, r)_{H_T}$ the ball $\{y\in H_T:\  \  \|x-y\|\leq r\}$ with center $x\in H_T$ and  radius $r>0$ in the hyperplane $H_T$. For a weight function $w$ on $T$,  we write $w(E)=\int_E w(x)\, dx$ for $E\subset T$, where $dx$ denotes the surface Lebesgue measure on $T$.

At the second step, we  prove

\begin{prop}\label{prop-1-10-0} Let
$T$ be a  surface simplex in $\RR^d$ spanned by a admissible subset  of  $\sph$ with $d\ge 3$, and let  $w$     be a weight  on  $T$  such that  $Nw(T)\in \NN$ for some $N\in\NN$. Assume  that  there exists $r_N\in (0, 1)$ such that $w(B)\ge \f 1N$ for every ball  $B=B(x,  r_N)_{H_T}\subset T$.
Then  there exists   a convex partition $\{ E_1, \cdots, E_{n_0}\}$ of $T$  such that  $Nw(E_j)\in \NN$ and $B(x_j, c_d'r_N)_{H_T}\subset E_j \subset B(x_j, c_d r_N)_{H_T}$ for some $x_j\in E_j$ and all $1\leq j\leq n_0$.
  \end{prop}

For the moment, we take Proposition \ref{prop-1-9-appendix} and Proposition \ref{prop-1-9-appendix} for granted and proceed with the proof of Theorem \ref{thm-partition}.
\begin{proof}[Proof of Theorem \ref{thm-partition}]  Assume that $\f N {2^{d-2}}\in\NN$, and set  $\da_0=\min_{x\in\sph} w(B(x, 10^{-d}))$.
We consider the following two cases:\\

{\bf Case 1:}\  \   $1\leq N\leq \da_0^{-1}$.\\

In this case,  $r\ge 10^{-d}$.  Thus,  by Lemma \ref{prop-2-2} and Lemma \ref{lem-2-6},
it suffices to prove the following assertion:  {\it given any integer $M\ge 1$ and any weight $w$ such that $w(B)>0$ for every spherical cap $B\subset \sph$,  there exists a convex partition $R_1, \cdots, R_M$ of $\sph$ with $w(R_j)=\f 1M$ for $1\leq j\leq M$.}
We prove this last assertion by  induction on the dimension $d$. If $d=2$, then the stated assertion follows directly  by continuity.
Now assume that the assertion holds  on the sphere $\SS^{d-2}$ with $d\ge 3$. Let
$w_{d-1}(\xi)=\int_0^\pi w(\cos \t, \xi\sin \t)\sin^{d-2}\t\, d\t$ for $\xi\in\SS^{d-2}$.
Clearly, $w_{d-1}(B)>0$ for every spherical cap $B$ in $\SS^{d-2}$.
 By  the induction hypothesis, there exists a convex  partition $E_1, \cdots, E_M$  of $\SS^{d-2}$ with $w_{d-1}(E_j)=\f 1M$ for $1\leq j\leq M$.
 Now set
$$ R_j:=\{ (\cos\t, \xi\sin\t):\   \ \xi\in E_j, \   \   0\leq \t\leq \pi\},\   \  j=1,2,\cdots, M.$$
It is easily seen that $\{R_1, \cdots, R_M\}$ is a convex partition of  $\sph$. Moreover, for each $1\leq j\leq M$,
\begin{align*}
 w(R_j)& = \int_0^\pi \int_{ E_j}w(\cos \t, \xi\sin \t)\, d\sa(\xi)\ \sin^{d-2}\t \, d\t
=w_{d-1}(E_j)=\f1N.\end{align*}
This completes the induction. \\

{\bf Case 2.}   $N\ge \da_0^{-1}$.\\

In this case, $ w(B(x, 10^{-d}))\ge \da_0\ge \f 1N$ for any $x\in\sph$.  Hence,
 applying Proposition \ref{prop-1-9-appendix}, there exists a partition   $\sph=\bigcup_{j=1}^m S_j$  of $\sph$ such that each $S_j$  is  an admissible geodesic simplex in $\sph$ with $Nw(S_j)\in\NN$.
 Let $T_j$ be the surface simplex in $\RR^d$ such that $g(T_j)=S_j$.
 Set $a_j:=\min_{x\in T_j}\|x\|$ and define
 $w_j(x)=w(g(x)) \f {a_j}{\|x\|^d}$ for $x\in T_j$. It's easily seen that  $\f 1{2d}\leq a_j\leq 1$. Moreover,
  according to Lemma \ref{prop-2-2} (iii) and (iv), $w_j$ is a weight on $T_j$ satisfying the following two conditions:  (a)
$w_j(E)=w(g(E))$ for each $E\subset T_j$,   and (b)  there exists a constant $r_N\sim_d r$ such that
$$ w_j \Bl(B(x, r_N)_{H_{T_j}}\Br)=w\Bl(g(B(x, r_N)_{H_{T_j}})\Br)\ge w(B(g(x), r))\ge \f 1N$$
whenever $B(x, r_N)_{H_{T_j}}\subset T_j$. Thus, applying  Proposition \ref{prop-1-10-0} to the weight $w_j$ on each surface simplex $T_j$,  we obtain a convex partition
$T_j=\bigcup_{i=1}^{n_j} E_{j,i}  $ such that
 $B(x_{j,i}, c_1r_N)_{H_{T_j}}\subset E_{j,i} \subset B(x_{j,i}, c_2 r_N)_{H_{T_j}}$ and $N w_j(E_{j,i}) \in \NN$ for all $1\leq i\leq n_j$.  Finally,  setting $R_{j,i} =g(E_{j,i})$ for $1\leq j\leq m$ and $1\leq i\leq n_j$, and applying Lemma \ref{prop-2-2} (ii),  we obtain a convex partition  $\{ R_{j,i}:\  \  1\leq j\leq m, \  1\leq i\leq n_j\}$  of  $\sph$  such that $Nw(R_{j,i})\in \NN$ and
$B(g(x_{j,i}), c_1'r_N)\subset R_{j,i} \subset B(g(x_{j,i}), c_2' r_N)$ for all $1\leq i\leq n_j$ and $1\leq j\leq m$.

\end{proof}

\subsection{Proof of Proposition \ref{prop-1-9-appendix}}
For simplicity,
 we say that  a set $E\subset \sph$ has a regular geodesic simplex partition with respect to the integer $N$ and the weight $w$ on $\sph$ if it can be written as a finite union of admissible geodesic simplexes $S_1, \cdots, S_n$ whose interiors are pairwise disjoint and such that $N w(S_j))\in \NN$ for each $1\leq j\leq n$.

 First, we claim that  Proposition  \ref{prop-1-9-appendix} is a consequence of the following assertion: \\

{\bf   (A)} \  \   {\it    The spherical cap $B(e_d, \f \pi2)\subset \sph$   has a regular geodesic simplex partition with respect to the integer $N$ and the weight $w$ on $\sph$ whenever $N$ and $w$ satisfy the following  conditions  (i) $ w(B(e_d, \f \pi2))=\f12$;  (ii) $\f N {2^{d-2}}\in\NN$;  and (iii) $N\ge
\max_{x\in\sph}\f 1{ w(B(x, 10^{-d}))}$, where $e_d=(0,0,\cdots, 0, 1)\in\sph$. }\\

To see this, let $w$ be a weight on $\sph$ satisfying the conditions of  Proposition  \ref{prop-1-9-appendix}.  Then $w(B(x, \f \pi 2))+w(B(-x, \f \pi 2))=w(\sph)=1$ for all  $x\in\sph$. It follows by continuity that
  $w(B(x_0, \f \pi 2))=w(B(-x_0, \f \pi 2))=\f12$ for some $x_0\in\sph$. Let  $\rho\in SO(d)$ be a rotation such that $\rho e_d=x_0$ and set $\wt{w}(x):=w(\rho x)$ for $x\in\sph$. Applying Assertion (A) to the weight $\wt{w}$ on $\sph$ and the integer $N$, we conclude that the spherical cap $B(x_0, \f \pi2)$   has a regular geodesic simplex partition with respect to  $N$ and  $w$.  A similar argument also yields the  same conclusion for the spherical cap $B(-x_0, \f \pi2)$. Since $\{ B(x_0, \f \pi2),  B(-x_0, \f \pi2)\}$ is a partition of $\sph$, it follows that  the sphere $\sph$  itself  has a regular geodesic simplex partition with respect to  $N$ and  $w$. This shows the claim.

     Assertion  A  can be proved using induction on the dimension $d$.
We start with the case of  $d=3$.  For simplicity, we write $\xi_\vi:=(\cos \vi, \sin\vi, 0)$
for $0\leq \vi\leq 2\pi$,   and  set $T(\al, \be):=\conv_{\SS^2} \{ e_3, \xi_\al, \xi_\be\}$ for $0\leq \al <\be \leq 2\pi$;  that is,
\begin{align*}
    T(\al, \be)&=\Bl\{ (\sin \t \cos\vi, \sin\t\sin\vi, \cos\t):\  \   0\leq \t\leq \f \pi 2,\    \  \al\leq  \vi\leq \be\Br\}.
\end{align*}
 By Lemma \ref{lem-6-2},  it is easily seen that  for any $\vi\in [0, 2\pi)$,
$$ T(\vi+\f \pi 6,\vi+ \f \pi 3)\supset B(x_0, \f \pi {40}),\   \  \text{with } \  \ x_0=\f {\sqrt{2}}2(e_3  + \xi_{\vi+\f \pi 4}),$$ which in particular implies that
\begin{equation}\label{6-2}
    w\Bl(T(\vi+\f \pi 6,\vi+ \f \pi 3)\Br)\ge \f 1N,\   \  \forall \vi\in [0,2\pi].
\end{equation}
By continuity,  we conclude  that for any $\vi\in [0, 2\pi]$, there exists $\al\in [\vi+\f \pi 6, \vi+\f \pi 3]$ such that $N w\Bl( T(\vi, \al)\Br)\in\NN$. Invoking  this  fact iteratively,  we
  obtain a sequence  of numbers  $\{\al_j\}_{j=0}^{k_0-1}$  such  that $\al_0=0$,
\begin{equation}\label{}
\f \pi 6 \leq \a_j-\a_{j-1}\leq \f \pi 3,\  \ \text{and}\  \
\   \    N w \bl( T(\al_{j-1}, \al_j)\br)\in \NN,\   \  1\leq j\leq k_0-1,
\end{equation}
   where $k_0\leq 12$ is the smallest   positive integer such that $\al_{k_0-1} +\f \pi 6< 2\pi$.
Setting $\al_{k_0}=2\pi$, we have  that
$\f \pi 6 < 2\pi -\al_{k_0-1}=\al_{k_0}-\al_{k_0-1}\leq \f \pi 2$.  Since $N$ is even, we have
$$\sum_{j=1}^{k_0} Nw\bl( T(\al_{j-1}, \al_j)\br)=Nw(B(e_3, \f \pi 2))=\f N2\in\NN.$$
 Thus,
$Nw\bl( T(\al_{k_0-1}, \al_{k_0})\br)\in \NN$.  Since $B(e_3, \f \pi 2)=\bigcup_{j=1}^{k_0} T(\a_{j-1}, \a_j)$,  we prove Assertion (A) for the case of $d=3$.

Now assume that Assertion A holds on the sphere $\sph$ for some $d\ge 3$, which in turn implies Proposition  \ref{prop-1-9-appendix} for   the sphere $\sph$.
 To show  Assertion (A)   on $\SS^d$,
assume that  $\f N {2^{d-1}}\in\NN$ and let  $w_d$ be  a weight on $\SS^d$ such that $w_d(B(e_{d+1}, \f \pi2))=\f12$ and  $N\ge
\max_{x\in\sph}\f 1{ w(B(x, 10^{-d-1}))}$.  For  $E\subset \sph$, we write,
$$S(E):=\Bl\{ ( x\sin \t, \cos \t):\  \  x\in E,\   \  0\leq \t \leq \f \pi 2\Br\}\subset \SS^d.$$
 Define $$ w_{d-1}(x)=2\int_0^{\f \pi 2} w_d( x\sin \t, \cos \t) \sin^{d-1}\t \, d\t,\  \  x\in\sph.$$ Clearly,  $w_{d-1}$ is a weight  on the sphere  $\sph$ satisfying that  for each $E\subset \sph$,
\begin{align*}w_d(S(E))&=\int_0^{\f \pi 2} \int_{E} w_d( x\sin \t, \cos \t)\, d\s_{d-1}(x) \sin^{d-1}\t\, d\t
=\f 12 w_{d-1}(E),\end{align*}
 In particular, $w_{d-1}(\sph)=2 w_d\Bl(B_{\SS^d}(e_{d+1}, \f \pi 2)\Br)=1$.
Next, setting   $\da_d=10^{-d-1}$ and  $N_{1}=N/2$,
we  claim that
\begin{equation}\label{6-4}
  \min_{x\in\sph}   w_{d-1} (B_{\sph} (x, \da_{d-1})) \ge  \f 1{N_1},
\end{equation}
where  the notation $B_{\SS^\ell} (x, r)$ is used for  the spherical cap in $\SS^\ell$.
For the proof of \eqref{6-4}, it is sufficient to show that for any $x\in\sph$,
\begin{equation}\label{6-5}
 B_{\SS^d}(y_x, \da_d) \subset S\Bl(B_{\sph} (x, \da_{d-1})\Br),
\end{equation}
$y_x=( x\sin \f \pi 4, \cos \f \pi 4)\in \SS^d$.
Indeed, once \eqref{6-5} is proved, then
  \begin{align*}
  w_{d-1} (B_{\sph} (x, \da_{d-1}))&=2w_d\Bl( S\Bl(B_{\sph} (x, \da_{d-1})\Br)\Br)\ge 2 w_d(B_{\SS^d}(y_x, \da_d))\\
  &\ge \f 2N =\f 1{N_1}. \end{align*}

 To show \eqref{6-5}, let
$z=( \xi\sin \t, \cos\t)\in B_{\SS^d} (y_x, \da_d)$ with  $\t\in [0,\pi]$ and $\xi\in\sph$.
Then by Lemma \ref{lem-6-2}, $|\t-\f \pi 4|\leq \dist(y_x, z)\leq \da_d$ (which implies $\f \pi 6\leq \t \leq \f \pi 3$), and
$$ \dist(\xi, x) \leq \f \pi 2 \f {\dist(y_x, z)}{\sqrt{\sin \t \sin \f \pi 4}}\leq \f \pi 2 \f {\da_d}{\sin\f \pi 6}=\pi \da_d\leq  \da_{d-1}.$$
This means  that $\xi\in B_{\sph}(x, \da_{d-1})$ and hence
$z=( \xi\sin\t, \cos\t)\in S\Bl( B_{\sph}(x, \da_{d-1})\Br)$.
 \eqref{6-5} then follows.

  Now applying  the induction hypothesis to the weight $w_{d-1}$ on $\sph$ and the integer $N_1$, we  conclude  that there exists  a regular geodesic simplex partition $\{E_1, \cdots, E_m \}$ of $\sph$  with respect to the integer  $N_1$ and the weight $w_{d-1}$ on $\sph$.
Setting  $S_j =S(E_j)$ for $1\leq j\leq m$, we obtain a partition $\{S_1, \cdots, S_m\}$ of the spherical cap  $B_{\SS^d}(e_{d+1}, \f \pi 2)$ satisfying that  for each $1\leq j\leq m$,
$$ N w_d(S_j) =N w_d(S(E_j))=N_1 w_{d-1}(E_j)\in\NN.$$
Thus, to complete the proof, it remains to show that each $S_j$ is an admissible  geodesic simplex on  $\SS^d$. Since each $E_j$ is an admissible geodesic simplex in $\sph$, there exists an admissible subset $\{ \xi_{j,1},\cdots, \xi_{j,d}\}$ of $\sph$ such that $E_j=\conv_{\sph} \{ \xi_{j,1},\cdots, \xi_{j,d}\}$ for each $1\leq j\leq m$.   Using  Lemma \ref{prop-2-2} (i), it is easily seen that
$$ S_j=S(E_j)=\conv_{\SS^d} \{  \wh{\xi_{j,1}},\cdots, \wh{\xi_{j,d}}, e_{d+1}\},$$
where $\wh{\xi_{j,i}}=(\xi_{j,i},0)\in\SS^d$ for $i=1,\cdots, d$.
Since  $\{  \wh{\xi_{j,1}},\cdots, \wh{\xi_{j,d}}, e_{d+1}\}$ is an admissible subset of $\SS^d$, we conclude that each $S_j$ is an admissible geodesic simplex of $\SS^d$. This completes the proof.

\begin{rem}\label{rem-2-7} If $w$ is a $\tau_d$-invariant weight on $\sph$ with $w(\sph)=1$, then the condition (i) of  Assertion (A) always holds.
\end{rem}
\begin{rem}\label{rem-2-8}  If  $w$ is a normalized  $\ZZ_2^d$-invariant weight on $\sph$, and $S=\conv_{\sph}\{ e_1,\cdots, e_d\}$, then $S_\tau:=\{ x\tau:\  \ x\in S\}$ is an admissible geodesic simplex with $w(S_\tau)=\f 1{2^d}$ for each $\tau\in\ZZ_2^d$.   Moreover, $\{ S_\tau:\  \ \tau\in\ZZ_2^d\}$ is a partition of $\sph$. Thus, Proposition ~\ref{prop-1-9-appendix} holds trivially if $w$ is $\ZZ_2^d$-invariant, $w(\sph)=1$ and $\f N {2^d}\in \NN$.
\end{rem}

   \subsection{Proof of Proposition \ref{prop-1-10-0}}
   Let $e_1=(1,0,\cdots, 0), \cdots, e_d=(0, \cdots, 0, 1)$ denote the standard canonical basis in $\RR^d$.
    Without loss of generality, we may assume that $T=\conv_{\RR^d}\{ \xi_1, e_2, \cdots, e_d\}$ with
    $\xi_1=(\sin\t, \cos\t, 0,\cdots, 0)$ for some $\f \pi 6\leq \t\leq \f \pi2$,
in which case $T$ can be written explicitly as
\begin{align*}
   T=&\Bl\{ x\in [0,\infty)^d:\  \  x_1 \tan \f \t2+\sum_{j=2}^{d} x_j =1 \ \ \text{and}\  \  x_{2} \tan \t \ge x_1\Br\}.
\end{align*}
For the rest, we  fix $\t\in [\f \pi 6, \f \pi 2]$, and  denote by $T_d$ the simplex in $\RR^d$ given by
$$ T_{d}:=\Bl\{ x\in [0, \infty)^d:\  \  x_1 \tan \f \t2+\sum_{j=2}^{d} x_j \leq 1 \ \ \text{and}\  \  x_{2} \tan \t \ge x_1\Br\}.$$
Also,  we will use the notation $B(x,r)_{\RR^d}$ to denote the Euclidean ball $\{y\in\RR^d:\  \ \|y-x\|\leq r\}$ in $\RR^d$.

 For the proof of Proposition \ref{prop-1-10-0}, we  claim that
 it is sufficient to show  the following assertion: \\

{\bf (B)} {\it If  $d\ge 2$, $N\in\NN$, $r\in (0,1)$ and   $w$  is    a weight  on   $T_d $ such  that   $Nw(T_d)\in \NN$, and    $\inf\bl\{ w(B): B=B (x, r)_{\RR^d}\subset T_d\br\}  \ge \f 1N$,    then $T_d$ has    a convex partition $\{ R_1, \cdots, R_{n_0}\}$  which satisfies that  $Nw(R_j)\in \NN$ and $B(x_j, c_d'r)_{\RR^d}\subset R_j \subset B(x_j, c_d r)_{\RR^d}$ for each  $1\leq j\leq n_0$ and some constants $c_d',c_d>0$.}\\

To show the claim, for $d\ge 3$, we
let $\Psi: T_{d-1}\to T$ denote  the   mapping given by $x \mapsto(x, \vi(x))$ with $\vi(x)=1-x_1\tan\f \t2-\sum_{j=2}^{d-1} x_j$ for $x\in T_{d-1}$.  Clearly,     $\Psi: T_{d-1}\to T$  is a bijective mapping that     maps convex sets to convex sets and satisfies $\|x-y\|\leq  \|\Psi(x)-\Psi(y)\|\leq C_d\|x-y\|$ for any $x,y\in T_{d-1}$. Moreover,  for each nonnegative function  $f:T\to [0, \infty)$, we have
\begin{equation*}
    \int_T f(x)\, dx =a_\t \int_{T_{d-1}} f\bl( \Psi(x)\br)\, dx,
\end{equation*}
where $a_\t:=\sqrt{d-1+\tan^2 \f \t 2}$. Thus, setting $w_{d-1}(x)=a_\t w(\Psi(x))$ for $x\in T_{d-1}$, we have that $w_{d-1}(E)=w(\Psi(E))$ for each $E\subset T_{d-1}$. In particular, this implies that $N w_{d-1}(T_{d-1})=Nw(T)\in\NN$
and $w_{d-1}(B) =w(\Psi(B))\ge \f 1N$ whenever $B=B(x,  r_N)_{\RR^{d-1}}\subset T_{d-1}$.
 Thus, applying Assertion~(B) to the integer $N$ and the weight $w_{d-1}$ on $T_{d-1}$, we get a partition $\{R_1,\cdots, R_{n_0}\}$ of $T_{d-1}$. It follows that  $\{E_j =\Psi(R_j):\   \   j=1,\cdots, n_0\}$ is a convex partition of the surface simplex $T$ with the stated properties in  Proposition \ref{prop-1-10-0}.
 This shows the claim.

  It remains to prove Assertion~ (B). We shall  use induction on the dimension  $d$.
  Let $L_d\ge 10$ be a large constant that  will be specified later.
  Let $r_1=L_d r$. Without loss of generality,
we may  assume that  $0<r\leq \f 1{100 L_d}$ since otherwise Assertion (B) holds trivially with $n_0=1$.

  We start with the case of $d=2$. Note that $T_2$ is an isosceles triangle  with vertices at $(0,0)$, $(0,1)$ and $(\sin\t, \cos\t)$ in the $x_1x_2$-plane. By rotation invariance of Assertion (B),
we may assume, without loss of generality,  that  $T_2$ is the triangle  with vertices at  $A=(-\sin \f \t2, \cos\f \t 2)$, $B=(\sin \f \t2, \cos\f \t 2)$ and $O=(0,0)$  in the $x_1x_2$-plane.
 For
$0\leq t<s\leq 1$, we set
$\tr(t,s):=\{  (ux,uy):\  \ (x,y)\in T_2,\ \ t\leq u\leq s\}$.
 It is easily seen that  for any  $0\leq s\leq 1-3r_1$,  the set $\tr(s+r_1, s+3r_1)$ contains a ball of radius $ r$, and hence $w(\tr(s+r_1, s+3r_1))\ge \f 1N$. By continuity, this implies that   for each  $s\in [0, 1-3r_1]$, there exists $t\in [s+r_1, s+3r_1]$ such that $N w(\tr (s, t))\in\NN$. Using this  fact iteratively,  we can construct a  partition  $0=t_0<t_1<\cdots< t_k=1$ of $[0,1]$ such  that
 $ t_{j-1}+r_1 \leq t_j \leq t_{j-1}+3r_1$   and  $N w(\tr (t_{j-1}, t_{j}))\in\NN$ for $1\leq j\leq k-1$,
  where $k$ is the largest   positive integer so that   $3r_1\leq 1-t_{k-1}\leq 6r_1$. Since $N w(T_2)\in\NN$, we also have that $N w(\tr (t_{k-1}, 1))\in\NN$.
  Note that for $0\leq j\leq k-1$,
  \begin{equation}\label{2-6:0}
    jr_1\leq t_j \leq 3r_1j\   \  \text{ and}\   \    jr_1  \leq t_j \|B-A\|\leq 3 jr_1.
  \end{equation}

  For simplicity, we set $\tr_j:=\tr(t_{j-1}, t_j)$ for $1\leq j\leq k$.
We  construct a  convex   partition of the domain   $\tr_j$ for each $10\leq j\leq k$  as follows.
   Let $A_j(s)= t_j A + t_j (B-A)s$,  $s\in [0,1]$ denote the parametric representation of the line segment from $t_jA$ to $t_j B$.
  For  $0\leq s<t\leq 1$, we denote by $T_{j}(s, t)$ the
     trapezoid with vertices at  $ A_{j-1}( s )$, $ A_{j-1}(t)$, $ A_j (s)$ and $ A_j (t)$.  It follows from \eqref{2-6:0} that for $0\leq s \leq 1-j^{-1}$
     $$ \text{area} \Bl(T_j(s, s+j^{-1})\Br) =\f 12(\cos\f \t2) (t_j-t_{j-1})j^{-1}\|A-B\|( t_j+t_{j-1})\sim r_1^2$$
     and $\diam(T_j(s, s+j^{-1}))\sim r_1$ .
   By Lemma \ref{lem-2-6}, this implies  that $T_j(s, s+j^{-1})$ contains a ball of radius $ c_2L_2r$. Thus, taking $L_2>c_2^{-1}$, we  conclude that  $w\Bl(T_j (s, s+j^{-1})\Br)\ge \f 1N$ for all $s\in [0, 1-j^{-1}]$. By continuity, this implies that for any $s\in [0, 1-3j^{-1}]$, there exists $t\in [s+j^{-1}, s+3j^{-1}]$ such that $N w(T_j (s, t))\in \NN$. Applying this fact iteratively, and recalling  that $Nw(\tr_j)\in\NN$,   we may construct   a partition $0=s_{j,0}<s_{j,1}<\cdots < s_{j, m_j}=1$ of $[0,1]$ such that
   $s_{j, i-1}+j^{-1}\leq s_{j,i} \leq s_{j,i}+3 j^{-1}$ for $1\leq i\leq m_j-1$, $j^{-1}\leq s_{j,m_j}-s_{j, m_j-1}\leq 6 j^{-1}$ and $N w\Bl(T_j (s_{j,i-1}, s_{j,i})\Br)\in \NN$ for $1\leq i\leq m_j$.
 Using Lemma \ref{lem-2-6}, we know that each set $T_j (s_{j,i-1}, s_{j,i})$ contains a ball of radius $r$ and has diameter $\sim r$.

 Now  putting the above together, we obtain a  convex partition of $T_2$:
$$\{ \tr_j:\  1\leq j\leq 9\} \cup\Bl\{ T_j \Bl(s_{j,i-1}, s_{j,i}\Br):\   \  10\leq j \leq k,\  \  1\leq i\leq m_j\Br\}.$$
Assertion (B) for $d=2$ then follows.

 Now assume that Assertion (B)  has been proven for the simplex $T_{d-1}\subset \RR^{d-1}$.  Let  $N$ and $w$ be a positive integer and   a weight on $T_d$ satisfying the conditions of Assertion (B).
   Note  that for any nonnegative function $f$ on $T_d$,
\begin{align}
    \int_{T_d}f(x)\, dx&=\int_0^1 \Bl[\int_{ T_{d-1}} f ( tx', 1-t)\, dx'\Br] t^{d-1}\, dt.\label{2-7:0}\end{align}
For $0\leq t<1-r_1$, define
    $$E_t:=\Bl\{ ( sx', 1-s):  x'\in T_{d-1},\   \   t\leq s\leq t+r_1\Br \}\subset T_d.$$
    It is easily seen that $E_t$ is a convex set in $\RR^d$  with diameter $\sim t+r_1$. Furthermore, by \eqref{2-7:0}, we have  $ \text{Vol}_d(E_t)\sim  (t+r_1)^{d-1} r_1$.  Thus,  according to Lemma \ref{lem-2-6},  for each $0\leq t<1-r_1$,  the set $E_t$  must contain a ball of radius at least $ c_d r_1=c_d L_d r$.  Taking the constant $L_d$ sufficiently large so that $c_d L_d>1$, we obtain from the assumption of Assertion (B) that that $w(E_t)\ge \f 1N$ for all $0\leq t\leq 1-r_1$.
   This together with the fact that $Nw(T_d)\in\NN$ implies that there is  a partition $0=t_0<t_1<\cdots <t_m=1$ of $[0,1]$ such that    $t_{j-1}+r_1\leq t_j \leq t_{j-1}+3 r_1$ for $1\leq j\leq m-1$, $ t_{m-1}+r_1\leq t_m\leq t_{m-1}+3r_1$ and
   $N w (A_j)\in \NN$ for $1\leq j\leq m$, where
   $A_j:=\Bl\{ ( sx', 1-s):  x'\in T_{d-1},\   \   t_{j-1}\leq s\leq t_j\Br \}$.
Next, for each $1\leq j\leq m$, we  define a weight function $w_j$ on $T_{d-1}$ by
   $$w_j(x'):=\int_{t_{j-1}}^{t_j}  w ( tx', 1-t)\,  t^{d-1}\, dt, \  \ x'\in T_{d-1}.$$
   We claim that for each  ball $B\subset T_{d-1}$ with radius $r_1$, $w_j(B)\ge \f 1N$. Indeed, setting
   $$\wt{B}_j:=\{ (sx', 1-s): x'\in B,\   \ t_{j-1}\leq s\leq t_j\},$$
   and using \eqref{2-7:0},
   we see that $\wt{B}_j$ is
  a convex subset of $T_d$ with diameter $\sim r_1$ and volume  $\sim   r_1^d$, By Lemma \ref{lem-2-6}, this implies that $\wt{B}_j$ contains a ball of radius $c_d r_1\ge r$. Using \eqref{2-7:0}, we obtain
$ w_j(B)=w(\wt{B}_j) \ge \f 1N.$
This shows the claim.

Now applying the induction hypothesis to the integer $N$, the radius $r_1$ and the  weights $w_j$ on $T_{d-1}$,  we conclude that for each $1\leq j\leq m$, the simplex   $T_{d-1}$ has a convex partition $R_{j,1},\cdots, R_{j,k_j}$ such that  $B(x_{j,i}, c_dr_1)_{\RR^{d-1}}\subset R_{j,i}\subset B(x_{j,i}, c_d'r_1)_{\RR^{d-1}}$ and $N w_j (R_{j,i})\in \NN$.
Setting
$$ T_{j,k}=\{ (sx', 1-s): x'\in R_{j,k},\   \ t_{j-1}\leq s\leq t_j\},$$
we get a convex partition $A_j =\bigcup_{i=1}^{k_j} T_{j,i}$ of the set $A_j$
satisfying that  $Nw(T_{j,k})=Nw_1(R_{j,k})\in\NN$,  $\diam(T_{j,k})\sim r_1$ and $\text{Vol}_d (T_{j,k})\sim r_1^d$.
This shows that $\{T_{j,i}:\  \ 1\leq j \leq m,\   \  1\leq i\leq k_j\}$ is a convex partition of $T_d$ with the stated properties in Assertion (B).

\subsection{Concluding remarks}

Combining Remarks \ref{rem-2-7} and \ref{rem-2-8} with Proposition \ref{prop-1-10-0}, we obtain
\begin{cor}
Let  $w$ be a weight on $\sph$ that is either   $\tau_d$-invariant or $\ZZ_2^d$-invariant and which satisfies  the conditions of Theorem \ref{thm-partition} for some  $r\in (0,\pi)$ and positive  even integer $N$. Assume in addition that $2^{-d} N$ is an integer if $w$ is $\ZZ_2^d$-invariant.
Let $S$ denote  either  the  semisphere $\{x\in\sph:\  \ x_d\ge 0\}$ or the geodesic simplex $\{x\in\sph:\  \ x_1,\cdots, x_d\ge 0\}$  according to whether  $w$ is $\tau_d$-invariant or $\ZZ_2^d$-invariant.
Then   $S$ has  a convex partition
 $\{R_1, \cdots, R_M\}$   for which the sets $R_j\subset S$ satisfy Conditions (i) and (ii) of Theorem \ref{thm-partition}.
\end{cor}

\section{Proof of Theorem \ref{thm-partition}:  the general case of  $N\in\NN$}

  Our goal in this section is to prove  Theorem \ref{thm-partition} without the extra condition $\f N {2^{d-2}}\in\NN$.  Assume that    $w$ is  a weight on $\sph$ normalized by $w(\sph)=1$ and  satisfying that
$w(B)>0$ for every spherical cap $B\subset \sph$. Let $r\in (0,1)$ and a positive integer $N$ satisfying that $\inf_{x\in\sph} w(B(x, r))\ge \f 1N$. We will keep these assumptions throughout this section.  For convenience, we introduce the following concept.
\begin{defn}\label{april1} Given a parameter  $\va\in (0,1)$, we say a subset $A=\{\eta_1, \cdots, \eta_m\}$ of $m$ distinct  points  ($m\leq d$) on  $\sph$ is strongly $\va$-separated if for each  $j=1,\cdots, m$,
$$\rho(\eta_j, X_j):=\inf_{x\in X_j}\|\eta_j-x\|\ge \va\   \   \
 \text{with}\  \   X_j=\text{span}\bl( A\setminus\{\eta_j\}\br).$$
   A  geodesic simplex in $\sph$ is said to be in the class $\mathcal{S}_\va$
        if it is spanned  by  a set of $d$ strongly $\va$-separated points $\eta_1,\cdots, \eta_d$ in $\sph$.
\end{defn}

 Theorem \ref{thm-partition} is a direct consequence of the following two propositions.

\begin{prop}\label{prop-1-10-appendix} Let $\va\in (0,1)$ be a given parameter and
$S$  a geodesic simplex  from the class $\mathcal{S}_\va$.    If $N w(S)\in\NN$,
    then there exists   a  convex partition $\{R_1,\cdots, R_n\}$ of   $S$   such that  $Nw(R_j)\in \NN$ and $B_{\sph}(x_j, c_{d,\va}r)\subset R_j \subset B_{\sph}(x_j, c_{d,\va}' r)$ for  each $1\leq j\leq n$.
  \end{prop}
  \begin{prop}\label{prop-1-13-appendix} There exists a  partition $\{S_1,\cdots, S_m\}$ of $\sph$ such that     each $S_j$ is a geodesic simplex from the class $\mathcal{S}_{\va_d}$ satisfying that $Nw(S_j)\in \NN$, where $\va_d$ is  a positive parameter   depending only on $d$.
\end{prop}

 Proposition \ref{prop-1-10-appendix} is  a  direct  consequence of Proposition \ref{prop-1-10-0}. For completeness, we write its proof below.
\begin{proof}[Proof of Proposition \ref{prop-1-10-appendix}]
Firstly, assume that $S$ is spanned  by  a set  of strongly $\va$-separated points $\eta_1,\cdots, \eta_d\in \sph$ for some $\va\in (0,1)$, and let $T=\conv_{\RR^d}\{\eta_1,\cdots, \eta_d\}$.  We claim that  $a:=\min_{x\in T}\|x\|\ge \va /d$. To see this, let  $x=\sum_{j=1}^d t_j \eta_j\in T$ be such that $a=\|x\|$, where $t_1, \cdots, t_d\ge 0$ and $\sum_{j=1}^d t_j=1$. Without loss of generality, we may assume that $t_1=\max_{1\leq j\leq d}t_j$. Then $t_1\ge \f 1d$, and
\begin{align*}
    a&=\|x\|=t_1\|\eta_1+\sum_{j=2}^d t_1^{-1}t_j \eta_j\|\ge \f 1d \rho(\eta_1, X_1)\ge \f \va d.
\end{align*}

Secondly,
let $A:=[\eta_1  \   \eta_2\   \cdots\  \eta_d]$ denote  the   nonsingular $d\times d$ real matrix   with column vectors $\eta_1,\cdots, \eta_d$.  We  assert  that
\begin{equation}\label{1-7-decomposition}
 \f {\va}{\sqrt{d}}\|x\| \leq   \|Ax\|\leq \sqrt{d}\|x\|,\  \ \forall x\in\RR^d.
\end{equation}
Without loss of generality, we may assume that $x\in\sph$ and  $|x_1|=\max_{1\leq j\leq d}|x_j|$. On the  one hand, by the Cauchy-Schwarz inequality,
\begin{align*}
    \|A x\|=\|\sum_{j=1}^d x_j \eta_j\|\leq \sum_{j=1}^d |x_j|\|\eta_j\|\leq \sqrt{d}.
\end{align*}
On the other hand, however, since $|x_1|\geq \f{1}{\sqrt{d}}$, we have
\begin{align*}
    \|Ax\|=|x_1|\Bl\|\eta_1+\sum_{j=2}^d \f {x_2}{x_1}\eta_j\Br\|\ge |x_1|\rho (\eta_1, X_1)\ge \f \va{\sqrt{d}}.
\end{align*}
This proves the inequality \eqref{1-7-decomposition}.
%

 Finally, we prove the assertion of Proposition \ref{prop-1-10-appendix}. Consider  the bijective  mapping $\Phi: x\mapsto g(Ax)$ from the surface simplex $T_0=\text{conv}_{\RR^d}\{e_1,\cdots, e_d\}$ to the geodesic  simplex $S$.  According to Lemma \ref{prop-2-2} (ii),   $\Phi$   maps convex sets to geodesic convex sets, and by
  \eqref{1-7-decomposition} and Lemma \ref{prop-2-2} (iv),
 \begin{equation}\label{3-2-0}
  c_{1,d}\va^2\|x-y\|\leq   \dist(\Phi(x), \Phi(y))\leq C_{1,d} \va^{-1}\|x-y\|,\   \  \forall x, y\in T_0.
 \end{equation} Since $\Phi$ maps boundary of $T_0$ to boundary of $S$, \eqref{3-2-0} in particular implies that if $B(x, t)_{H_{T_0}}\subset T_0$ for some $x\in T_0$ and $t\in (0,1)$, then $$B_{\sph} (\Phi(x), c_{1,d}\va^2 t)\subset \Phi\Bl( B(x, t)_{H_{T_0}}\Br)\subset S. $$
  On the other hand,
using  Lemma \ref{prop-2-2} (iii),  we obtain by  a change of variable    that
 for each nonnegative function $f$ on $S$,
 \begin{equation}\label{3-2-00}
   |\text{det} A|  \int_{T_0} f(\Phi(x)) w(\Phi(x))\f a {\|A x\|^d}\, dx =\int_{S} f(z) w(z)\, d\s (z).
 \end{equation}
 Thus, setting
 $w_0(x)=a|\text{det} A|  w(\Phi(x)) / \|A x\|^d$ for $x\in T_0$, we have   $N w_0(T_0)=N w(S)\in\NN$ and
$$\inf\Bl\{ w_0(B):\  \  B=B(x, c r)_{H_{T_0}}\subset T_0\Br\}\ge \f 1N$$
with $c=1/c_{1,d}\va^2$.
     Thus, according to Proposition \ref{prop-1-10-0}, $T_0$ has a convex partition $T_0=\bigcup_{j=1}^m T_{0,j}$ with the properties that
$ B(x_j, c_{d} r)_{H_{T_0}}\subset T_{0,j}\subset B(x_j, c_{d}' r)_{H_{T_0}}$ and $N w_0(T_{0,j})\in\NN$ for all
$1\leq j\leq m$. Setting
$$ R_j=\Phi(T_{0,j}) =\{ g(A x):\  \ x\in T_{0,j}\}, \  \  j=1,2,\cdots, m,$$
and using \eqref{3-2-00} and Lemma \ref{prop-2-2}, we conclude that  $\{R_1, \cdots, R_m\}$ is a convex partition of the geodesic simplex  $S$ with the stated properties  in Proposition \ref{prop-1-10-appendix}.

\end{proof}

 Finally, we turn to the proof of Proposition ~\propref{prop-1-13-appendix}, which  is more involved.  The important ingredient used in the proof is a  family of nonlinear dilations on $\sph$, which we shall introduce and study in the subsection that follows.
\subsection{A family of nonlinear dilations}
Throughout this subsection, we write   $ \mathbf{x}(\vi, \xi)=(\cos \vi, \xi\sin\vi)\in \sph$
for $(\vi, \xi)\in [0,\pi]\times \SS^{d-2}$.
\begin{defn}\label{defn-3-4:2017}
For $0\leq \a<\b\leq 2\pi$,  define
\[S(\a,\b):=\Bl\{x=(r\cos \t, r\sin\t, \sqrt {1-r^2} \eta):\  \  \eta\in\SS^{d-3},\  \  0\leq r\leq 1,  \al\leq \t\leq \be\Br\},\]
where it is agreed that $\SS^0=\{ \pm 1\}$ when $d=3$.
\end{defn}

  Next,   for  $0<\al<\f\pi2$, we define the function $h_\al:  [-1,1]\to (0,1]$  by
$$ h_\al (t):= \begin{cases}\f 2\pi \tan^{-1} \Bl( \f{\tan \al}{|t|}\Br),\ \ &\text{if $0<|t|\leq 1$;}\\
 1,\   \  & \text{if $t=0$}.\end{cases}$$
It is easily seen that   $\f 2\pi \al \leq h_\al(t)\leq 1$ for $t\in [-1,1]$ and
\begin{equation}\label{derivative}
  h_\al'(t)=  -\f 2\pi \f{\tan \al}{t^2 +\tan^2 \al},\  \  t\in [-1,1].
\end{equation}
Using the function $h_\al$, we may define  a nonlinear  dilation $T_\al x$ of $x\in\sph$ with respect to the angle between $x$ and $e_1$ as follows:
  \begin{defn}\label{def-1-14:convex}
  Given  $0<\al<\f\pi2$,
 define the mapping   $T_\al:  \SS^{d-1}\setminus\{-e_1\}\to \SS^{d-1}\setminus\{-e_1\} $  by
$T_\al (x):=\mathbf{x}
\bl(h_\al(\xi_1) \vi, \xi\br)$  for $  x=\mathbf{x}(\vi,\xi)\in \SS^{d-1}\setminus\{-e_1\}$
 with $ \vi\in [0, \pi)$ and  $\xi=(\xi_1,\cdots, \xi_{d-1})\in\SS^{d-2} $.
\end{defn}

 The following lemma collects  some useful properties of $T_\al$:

\begin{lem}\label{lem-3-6:2017} Let  $0<\al<\f \pi2$ and let $\SS_0^{d-1}:=\{x\in\SS^{d-1}:\ \  x_1=0,\  x_2\ge 0\}$. Then the following statements hold:
\begin{enumerate}[\rm (i)]
\item  If  $x, y\in\sph$ and $0\leq \dist(x, e_1), \dist(y, e_1)\leq \f {3\pi} 4$, then
    $$ c_{d, \al} \dist(x,y)\leq \dist(T_\al x, T_\al y)\leq C_{d,\al}  \dist(x,y).$$

\item  $T_\al$ is a bijective mapping  from $S(0,\f \pi2)$ to $S(0,\al)$ that maps boundary of $S(0,\f \pi2)$ to the boundary of $S(0,\al)$.
  \item  For every nonnegative measurable function $f$  on $S(0,\al)$,
  \begin{align}
    \int_{S(0,\al)} f(x)\, d\s_d(x) =\int_{S(0, \pi/2)} f(T_\al x) \og(x)\, d\s_d(x),\label{1-3-appendix}
  \end{align}
  where
  \begin{equation}\label{1-12:convex}
     \og(x):=\Bl( \f {1-(T_\al (x)\cdot e_1)^2 }{1-x_1^2}\Br)^{\f{d-2}2} h_\al\Bl(\f{x_2}{\sqrt{1-x_1^2}}\Br),\   \  x\in S(0, \f \pi2).
  \end{equation}
  \item
  If  $S\subset S(0, \f \pi2)$  is  geodesic simplex  spanned by  $e_1$ and a set  of independent vectors  $\eta_1, \cdots, \eta_{d-1}\in\SS_0^{d-1}$, then  the set $T_\al(S)\subset S(0,\al)$ is a geodesic simplex  spanned by the set $\{T_\al(\eta_1), T_\al(\eta_2),\ldots, T_\al(\eta_{d-1}), e_1\}$.
  \item  If $\{\eta_1, \cdots, \eta_{d-1}\}$ is a set of  strongly $\va$-separated points in $\SS_0^{d-1}$ for some $\va\in (0,1)$, then $\{ T_\al (\eta_1),\cdots, T_\al (\eta_{d-1}), e_1\}$ is a set of  strongly $c_{d,\al} \va$-separated points on $\sph$.

\end{enumerate}
\end{lem}

\begin{proof}  (i) Assume that   $x=\mathbf{x}(\t_1, \xi)$ and $y=\mathbf{x}(\t_2, \eta)$ with $0\leq \t_1\leq \t_2\leq \f{3\pi}4$. According to Lemma \ref{lem-6-2}, we have
 \begin{equation}\label{2-100}
  \dist(x,y) \sim |\t_1-\t_2|+\sqrt{\t_1\t_2}\dist(\xi,\eta).
  \end{equation}
  and
\begin{align}
  \dist(T_\al x,
  T_\al y)| \sim |h_\al(\xi_1)\t_1-h_\al(\eta_1)\t_2|+
  \sqrt{\t_1\t_2}\dist(\xi,\eta).\label{3-8-2017}
  \end{align}
   By \eqref{derivative},
\begin{align*}
    |h_\al(\xi_1)\t_1-h_\al(\eta_1) \t_2|&\leq \t_1 |h_\al(\xi_1)-h_\al(\eta_1)|+|\t_1-\t_2|\\
    &\leq  C\sqrt{\t_1\t_2} \dist(\xi, \eta)+|\t_1-\t_2|,
\end{align*}
which combined with \eqref{2-100} and \eqref{3-8-2017} implies that
$\dist(T_\al x, T_\al y)\leq C_{d,\al}\dist(x,y)$.
On the other hand, a similar argument shows that
\begin{align*}
|\t_1-\t_2|&\leq \Bl|\f 1{h_\al(\xi_1)}-
\f 1 {h_\al(\eta_1)}\Br|h_\al(\xi_1) \t_1 +
   \f1{ h_\al(\eta_1)}\Bl| h_\al(\xi_1)\t_1 -h_\al (\eta_1)\t_2\Br|\\
    &\leq C_{d,\al} \Bl(\sqrt{\t_1\t_2} \dist(\xi,\eta)+\Bl| h_\al(\xi_1)\t_1 -h_\al (\eta_1)\t_2\Br|\Br),\end{align*}
    which, using \eqref{2-100} and \eqref{3-8-2017} once again, implies the inverse inequality $\dist(x,y)\leq \dist(T_\al x, T_\al y)$.   This completes the proof of (i).

  (ii) Since  $(\vi,\xi)\mapsto \mathbf{x}(\vi,\xi)$ is an injective mapping from $(0, \pi)\times \SS^{d-2}$ to $\sph$, it follows  that $T_\al : \sph\setminus\{-e_1\}\to \sph\setminus\{-e_1\}$ is injective. To show that $T_\al$ is a mapping from $S(0, \f \pi2)$ onto $S(0,\al)$,  we need the following fact, which is a direct consequence of  Definition \ref{defn-3-4:2017}:

  {\bf Fact 1.}\  \   If $0<\al \leq  \f \pi 2$,
then $\mathbf{x}(\vi, \xi)\in S(0,  \al)$ if and only if $\vi=0$ or   $0< \vi\leq  \f \pi 2$, $\xi_1\ge 0$ and $\tan \vi \leq \f {\tan \al}{\xi_1}$, where is agreed that $\tan \f \pi2=\infty$ and $\f {\tan \al}{0}=\infty$.

  Note also that according to Definition \ref{def-1-14:convex},    if   $x=(x_1, x_2,\cdots, x_d)\in \sph\setminus\{-e_1\}$ with  $x_2=0$, then $T_\al x=x$, which in particular implies  $T_\al (e_1)=e_1\in S(0,\al)$.
  If $x=\mathbf{x}(\vi, \xi)\in S(0, \pi/2)\setminus\{e_1\}$, then  $0< \vi\leq \f \pi2$,  $\xi_1\ge 0$ and
$
    \tan (h_\al(\xi_1) \vi)\leq  \tan (\f \pi 2 h_\al(\xi_1))=\f {\tan \al}{\xi_1},
$
 which, by Fact 1, implies that  $T_\a (x)\in S(0,\al)$.   Conversely, if  $y=\mathbf{x}(\vi_1, \xi)\in S(0,\al)\setminus\{e_1\}$, then by Fact 1,  $0< \vi_1\leq \f\pi2$, $0\leq \xi_1\leq 1$,  $\tan (\vi_1)\leq \f {\tan \al}{\xi_1}$,
and hence
\begin{align*}
   0<  \f{\vi_1}{h_\al(\xi_1)} \leq \f 1{h_\al(\xi_1)} \tan^{-1}\Bl(\f {\tan \al}{\xi_1}\Br)=\f \pi2.
\end{align*}
This implies that $x=\mathbf{x}(\f{\vi_1}{h_\al(\xi_1)}, \xi)\in S(0, \f \pi2)$  and $y=T_\al(x)$. Finally, we show $T_\al$ maps boundary of $S(0, \f \pi2)$ to boundary of $S(0,\al)$. Indeed,
if $x=(r, 0 \sqrt{1-r^2}\eta)\in S(0, \f \pi2)$ with $0\leq r\leq 1$, then $T_\al x =x\in S(0,\al)$.
On the other hand, it can easily verified that
 if $x=(0, \xi)\in S(0, \f \pi2)$ with $\xi_1>0$, then $T_\al x= (r\cos \al, \f {(r\sin\al) \xi}{\xi_1}) \in S(0,\al)$ with $r=\f {\xi_1}{ \sqrt{\xi_1^2\cos^2\al +\sin^2\al}}$.

(iii)  Using Fact 1, we  obtain
\begin{align*}
   \int_{S(0,\al)} f(x)\, d\s_d(x)&=\int _{\sub{\xi\in \SS^{d-2}\\
    \xi_1\ge 0}} \int_{0\leq \vi\leq\f \pi 2 h_\al(\xi_1)} f(\mathbf{x}(\vi,\xi))\sin^{d-2}\vi \, d\vi \, d\s_{d-1}(\xi), \end{align*}
   which, by  a change of variable $\vi=h_\al(\xi_1) \t$, equals
   \begin{align*}
   &\int_{\sub{\xi\in \SS^{d-2}\\
   \xi_1\ge 0}} \int_0^{\f\pi 2} f(T_\al (\mathbf{x}(\ta,\xi)))   {\sin^{d-2}(h_\al(\xi_1)\ta)}\, d\ta h_\al(\xi_1)\, d\s_{d-1}(\xi)
   =\int_{S(0, \pi/2)} f(T_\al x) \og(x)\, d\s_d(x).
\end{align*}

(iv) We start with the proof of the following fact:

{\bf Fact 2.}\  \  If $x=(0, \bar{x})\in \sph_0$,  then $T_\al x =\f {E_\al (x)}{\|E_\al (x)\|},$
         where
         $E_\al: \RR^d\to \RR^d$ is a linear operator given by
         $ E_\al y =(y_2\cos\al, \bar{y}\sin \al)$ for $y=(y_1, y_2,\cdots, y_d)=(y_1, \bar{y})\in\RR^d$.

 To see this, let $x=(0, x_2,\cdots, x_d)\in\sph_0$.
 If $x_2=0$, then $T_\al x=x$ and  Fact 2  holds trivially. If $x_2>0$, then  setting  $h=\f \pi 2 h_\al (x_2)=\tan^{-1}\f{\tan \al}{x_2}$, we have
 \begin{align*}
T_\al (x) &= (\cos h, (\sin h) \bar{x})=\f {\cos h}{x_2}( x_2, (\tan \al )\bar{x})
=\f{E_\al (x)}{\|E_\al (x)\|}.
\end{align*}
This proves Fact 2.

Next,  we prove that for  $V=\conv_{\mathbb{S}^{d-1}}\{ \eta_1, \cdots, \eta_{d-1}\},$
  \begin{equation}\label{1-appendix}
     T_\al (V)=\conv_{\mathbb{S}^{d-1}}\{ T_\al(\eta_1), \cdots, T_\al(\eta_{d-1})\}.
  \end{equation}

Indeed,
  if $x\in V$, then $x=\sum_{j=1}^{d-1} t_j \eta_j\in \SS_0^{d-1}$ for some $t_1, \cdots, t_{d-1}\ge 0$, and hence, by Fact 2,
  \begin{align*}
    T_\al (x) =\f{E_\al (x)}{\|E_\al(x)\|} =\sum_{j=1}^{d-1} \f{t_j}{\|E_\al(x)\|} E_\al (\eta_j)=\sum_{j=1}^{d-1} \f{t_j\|E_\al (\eta_j)\|}{\|E_\al(x)\|} T_\al (\eta_j),
  \end{align*}  which  implies that $T_\al (x)\in \conv_{\sph}\{T_\al(\eta_1),\ldots, T_\al(\eta_{d-1})\}$.
Conversely, if $y\in \conv_{\sph}\{T_\al(\eta_1),\ldots, T_\al(\eta_{d-1})\}$, then
  $y=\sum_{j=1}^{d-1} s_j T_\al (\eta_j) $ for some $s_1, \cdots, s_{d-1}\ge 0$, and by Fact 2,
  \begin{align*}
    y& =E_\al \Bl(\sum_{j=1}^{d-1} \f{s_j \eta_j} {\|E_\al (\eta_j)\|}\Br)
    =  c_1 E_\al \Bl(\f{\sum_{j=1}^{d-1} s_j' \eta_j}{\|\sum_{j=1}^{d-1} s_j' \eta_j\|}\Br)=T_\al  \Bl(\f{\sum_{j=1}^{d-1} s_j' \eta_j}{\|\sum_{j=1}^{d-1} s_j' \eta_j\|}\Br)\in T_\al(V),
  \end{align*}
  where $s_j'=\f{s_j } {\|E_\al (\eta_j)\|}$ and
  $c_1=\|\sum_{j=1}^{d-1} s_j' \eta_j\|$.

  Third, we show that for any $\eta\in \SS_0^{d-1}$,
  \begin{equation}\label{1-4-appendix}
    T_\al(\Arc (e_1, \eta))=\Arc(e_1, T_\al(\eta)).
  \end{equation}

 Assume that $\eta=(0, \xi)$ with $\xi=(\xi_1,\cdots, \xi_{d-1})\in\SS^{d-2}$ and $\xi_1\ge 0$.    Then
$\Arc(e_1, \eta)=\Bl\{ (\cos \t, \xi \sin\t):\  \  \t\in [0,\f \pi2]\Br\}.$
 If $x=(\cos \vi, \xi\sin \vi)\in \Arc(e_1, \eta)$ with $\vi\in [0, \f \pi2]$, then setting $h_\al =\f 2\pi \tan^{-1} \f{\tan \al}{\xi_1}$, we have
\begin{align*}
    T_\al (x)&=\f{\sin (h_\al \vi)}{\sin \al} E_\al (\eta) + \Bl(\cos (h_\al \vi)-\f {\sin (h_\al \vi)}{\tan \al}\xi_1\Br) e_1
    =s_1 T_\al (\eta) + s_2 e_1,
\end{align*}
where $s_1=\f{\sin (h_\al \vi)}{\sin \al}\|E_\al(\eta)\|\ge 0$ and $s_2 =\cos (h_\al \vi)-\f {\sin (h_\al \vi)}{\tan \al}\xi_1$.
Since
$\tan (h_\al \vi) \leq \tan (\f \pi 2h_\al )= \f {\tan \al}{\xi_1},$
 it follows that $s_2\ge 0$, and hence  $T_\al x\in \Arc(e_1, T_\al(\eta))$.
Conversely, if $y=t_1 e_1+t_2 T_\al (\eta)\in \Arc (e_1, T_\al (\eta))$ for some $t_1, t_2\ge 0$, then
\begin{align*}
    y =t_1 e_1 + \f {t_2E_\al (\eta)}{\|E_\al(\eta)\|}  =(t_1 + \f {t_2\xi_1\cos\al}{\|E_\al(\eta)\|} ) e_1 + \f {t_2(\sin \al) \eta}{\|E_\al(\eta)\|} \in \Arc(e_1, \eta).
\end{align*}

Finally, we prove \begin{equation}\label{1-4-0-appendix}
    T_\al(S)=\conv_{\mathbb{S}^{d-1}}
\{T_\al(\eta_1), T_\al(\eta_2),\ldots, T_\al(\eta_{d-1}), e_1\}.
 \end{equation}
Indeed,   it is easily seen that each  $x\in S$ can be written  in the form  $x=\sqrt{1-t^2}\eta +t e_1\in \Arc (e_1, \eta)$ for some $t\ge 0$ and  $\eta\in V=\conv_{\mathbb{S}^{d-1}}\{ \eta_1, \cdots, \eta_{d-1}\}$. It follows by \eqref{1-4-appendix} that
$ T_\al (x) \in \Arc(e_1, T_\al (\eta))$. However, by \eqref{1-appendix},
$$T_\al (\eta)\in T_\al(V)= \conv_{\sph}\{ T_\al(\eta_1),\cdots, T_\al(\eta_{d-1})\}.$$
Thus,
$$ T_\al (x) \in \Arc(e_1, T_\al (\eta))\subset
\conv_{\sph}\{ T_\al(\eta_1),\cdots, T_\al(\eta_{d-1}), e_1\}.$$
Conversely, if $y\in \conv_{\sph}\{ T_\al(\eta_1),\cdots, T_\al(\eta_{d-1}), e_1\}$, then
$y=s_0e_1+\sum_{j=1}^{d-1} s_j T_\al(\eta_j)$ for some $s_0, s_1,\cdots, s_{d-1}\ge 0$.  According to \eqref{1-appendix},
there exists $\eta\in V$ such that
$\sum_{j=1}^{d-1} s_j T_\al(\eta_j)=c_2 T_\al(\eta),$
with $c_2=\|\sum_{j=1}^{d-1} s_j T_\al(\eta_j)\|$. It then  follows by \eqref{1-4-appendix} that
$$ y \in \Arc (e_1, T_\al (\eta)) =T_\al \Bl( \Arc (e_1, \eta)\Br)\subset T_\al (S).$$

 (v)\   \
According to Fact 2,   we have
\begin{align*}
   &\min_{t_1,\cdots, t_{d-1}\in\RR}
    \Bl\|e_1 -\sum_{j=1}^{d-1} t_j T_\al (\eta_j) \Br\|=
    \min_{t'_1,\cdots, t'_{d-1}\in\RR} \Bl\|e_1 -E_\al \Bl(\sum_{j=1}^{d-1} t_j' \eta_j\Br)\Br \|\ge \min_{y\in\RR^{d-1}} \Bl\|e_1 -E_\al \Bl((0,y)\Br)\Br \|\\
    &= \min_{ y\in\RR^{d-1}}\|(1-y_1\cos \al, y\sin \al)\|\ge \min_{ y\in\RR^{d-1}}\max \Bl\{ 1-\|y\|\cos\al, \|y\|\sin\al\Br\}\ge \f { \sin \al}2.
\end{align*}
By symmetry, it remains to show that
\begin{align}
    \min_{t_2,\cdots, t_{d}\in \RR}\Bl \|T_\al (\eta_1)-\sum_{j=2}^{d-1} t_j T_\al (\eta_j)-t_d e_1\Br\|\ge c_{d,\al}\va.\label{3-12:2017}
\end{align}
Indeed, since $\sin \al \leq \|E_\al (\eta_1)\|\leq \sqrt{2}$, we obtain from  Fact 2 that
\begin{align*}
  &\text{LHS of \eqref{3-12:2017}} \ge \f 1{\sqrt{2}}  \min_{t_2,\cdots, t_{d}\in \RR}\Bl\|E_\al \bl(\eta_1 -\sum_{j=2}^{d-1} t_j \eta_j\br)-t_d e_1\Br\|\ge \f 1{\sqrt{2}}  \min_{\sub{y\in\RR^{d-1},\  \|y\|\ge \va\\
  t\in\RR}}\Bl\|E_\al \bl((0,y)\br)-t e_1\Br\|\\
  &= \f 1{\sqrt{2}}  \min_{\sub{ y\in\RR^{d-1},\ \|y\|\ge \va_d\\
  t\in\RR}}\Bl\|(y_1\cos\al -t, y\sin\al) \Br\|\ge \f {\va_d \sin\al} {\sqrt{2}}.
\end{align*}

\end{proof}

\subsection{Proof of Proposition \ref{prop-1-13-appendix}}  Without loss of generality, we may assume that $N\ge N_{d,w}$.
Given $0\leq \al<\be \leq  2\pi$, let  $S(\al,\be)$ denote the   subset of $\sph$ given  in Definition \ref{defn-3-4:2017}.
The following assertion plays an important role in the proof of Proposition \ref{prop-1-13-appendix}.\\

 {\bf Assertion A.}\  \  If  $N w(S(0,\f\pi2))\in\NN$, then  there exists  a convex partition $\{ S_{1}, \cdots, S_{n_0}\}$ of the set $S(0,\f\pi2)$ such that each $S_j$ is a geodesic simplex spanned by the vector $e_1$ and a set of $d-1$ strongly $\va_d$-separated points in  $\SS_0^{d-1}=\{x\in\sph:\  \ x_1=0,\  \ x_2\ge 0\}$, and satisfying the condition $N w(S_j)\in\NN$.\\

For the moment, we take Assertion A for granted and proceed with the proof of the proposition. Note that if  $0\leq \al <\be \leq 2\pi$ and  $\be-\al\ge \f \pi6$, then $S(\al,\be)$ contains a  spherical cap of radius $\ge \f 1{20}$. It follows by continuity  that $\sph$ has a partition   $\sph=\bigcup_{j=1}^{\ell_0} S(\al_{j-1},\al_j)$ with    $1\leq \ell_0\leq 12$ such that  $\al_0=0$,   $\f \pi 6\leq \al_i-\al_{i-1}< \f \pi2$  and $Nw(S(\al_{i-1},\al_i))\in\NN$ for $i=1,\cdots, \ell_0$.
Thus,  replacing  $w$ with $w\circ Q$ for some rotations $Q\in SO(d)$, we reduce  to  proving  the following assertion:\\

{\bf Assertion B:}\  \ If $\al\in [\f \pi6, \f \pi 2]$ and $N w(S(0,\al))\in\NN$, then  there exists  a convex partition $\{ S_{1}, \cdots, S_{n_0}\}$ of the set $S(0,\al)$ such that  $S_{j}\in \mathcal{S}_{\va_d}$ and $N w(S_j)\in\NN$ for $j=1,\cdots, n_0$.  \\

The proof of Assertion B relies on Assertion A and the nonlinear bijective mapping $T_\al : S(0, \f\pi2)\to S(0, \al)$ introduced  in Definition \ref{def-1-14:convex}.
 Let $w_1(x):=w(T_\al(x))\og(x)$ for $x\in S(0, \f \pi2)$, where $\og$ is the function given in \eqref{1-12:convex}.
  According to \eqref{1-3-appendix}, $w_1(E)=w(T_\al(E))$ for each $E\subset S(0, \f\pi2)$.  Moreover, by By Lemma \ref{lem-3-6:2017} (i), (ii) and (iii), we may apply  Assertion A  to the weight $w_1$ on $S(0, \f \pi2)$ and  obtain  a partition  $\{V_1,\cdots, V_{n_0}\}$ of $S(0, \f \pi2)$ with the stated properties of Assertion A (with  $w_1$ in place of  $w$). Now setting  $S_j=T_\al(V_j)$, we get a partition $\{S_1, \cdots, S_{n_0}\}$ of $S(0,\al)$.
   By Lemma \ref{lem-3-6:2017} (iv) and (v),  each $S_j$ is a geodesic simplex in the class $\mathcal{S}_{\va_d}$,  whereas by Lemma \ref{lem-3-6:2017} (iii),
    $N w(S_j)=Nw_1(T_j)\in\NN$ for $1\leq j\leq n_0$.  This  proves Assertion B.

It remains to show Assertion A.
 We start with the case of $d=3$. Note that
 $$S(0, \f \pi 2) =\Bl\{ (\cos \vi, \sin\vi \sin\t, \sin\vi \cos\t):\  \  0\leq \vi\leq \f \pi2,\  \  0\leq \t\leq \pi\Br\}.$$
  By continuity,  there exists a point $\xi=(0, \sin\t, \cos\t)\in \SS_0^{2}$ for some $\t\in (\f \pi6, \f \pi2)$  such that $S(0, \f \pi2)=S_1\cup S_2$ with $S_1: =\conv_{\SS^2} \{ e_1, -e_3, \xi\}$ and  $S_2: =\conv_{\SS^2} \{ e_1, e_3, \xi\}$, and such that
 $Nw(S_1), N w(S_2)\in\NN$. This shows Assertion A and hence Proposition \ref{prop-1-13-appendix} for $d=3$.

   Next, we show Assertion A for $d\ge 4$.
   We  use induction on the dimension $d$. Assume that the conclusion of Proposition \ref{prop-1-13-appendix}  holds on    the spheres $\SS^{\ell-1}$,  $\ell=3,\cdots, d-1$. We shall prove that Assertion A holds for $d\ge 4$, which in turn  implies Proposition \ref{prop-1-13-appendix} for $d\ge 4$.
   The proof relies on the following formula:  \begin{align}
    &\int_{S(0, \pi/2)} f(x)\, d\s(x)
     = \int_{\SS^{d-3}} \int_0^1 \int_{0}^{\pi/2} f(r\cos\t, r\sin \t, \sqrt{1-r^2}\xi) d\t (1-r^2)^{\f {d-4}2}\, dr \, d\s(\xi),\label{1-16:convex}
\end{align}
where $f$ is  a nonnegative function   on $S(0, \pi/2)$.

For $\eta\in\SS^{d-3}$,
define
$$ \wt{w}(\eta):=\f{1}{\og}\int_0^{\f \pi2} \int_0^1 w(r\cos\t, r\sin\t, \sqrt{1-r^2}\eta)(1-r^2)^{\f {d-4}2}\, dr \, d\t,$$
where $\og=w(S(0, \pi/2))=\f {k_0}N$ for some positive integer $k_0$.
Given a set $E\subset \SS^{d-3}$, set
 $$ \wh{E}:=\Bl\{ (x_1, x_2, \sqrt{1-x_1^2-x_2^2} \eta):\  \ x_1, x_2\ge 0,\  \ x_1^2+x_2^2\leq 1,\   \ \eta\in E\Br\}.$$
   It can be easily seen that if  $E$ is a geodesic simplex  in $\SS^{d-3}$ spanned by a set  of linearly independent points $\eta_{1},\cdots, \eta_{d-2}\in\SS^{d-3}$,   then $\wh{E}$  is a geodesic simplex in $\SS^{d-1}$ spanned by the set
  $\{e_1, e_2, \wh{\eta}_{1}, \cdots, \wh{\eta}_{d-2}\}$,
 with $\wh{\eta}_{j} =(0,0, \eta_{j})\in\SS_0^{d-1}$ for $j=1,\cdots, d-2$, and moreover,
by \eqref{1-16:convex},
$\wt{w}(E)=\og^{-1} w(\wh{E})$.
Now applying the induction hypothesis to the weight $\wt{w}$ on the sphere $\SS^{d-3}$ with $k_0$ in place of $N$, we obtain a  partition  $\{E_1, \cdots, E_{m_0}\}$ of the sphere   $\SS^{d-3}$, where each $E_j$ is a geodesic simplex in $\SS^{d-3}$ spanned by a  set of  strongly $\va_{d-2}$-separated points $ \eta_{j, 1},\cdots, \eta_{j, d-2}\in \SS^{d-3}$  and satisfying
 $k_0 \wt{w}(E_j)\in\NN$.   It then follows that  $\{ \wh{E}_1,\cdots, \wh{E}_{m_0}\}$ is  a  partition of $S(0, \f \pi2)$ where each $\wh{E}_j$ is a geodesic simplex spanned by the vector $e_1$ and the  set of strongly $\va_d$-separated points $\wh{\eta}_{j,1},\cdots, \wh{\eta}_{j,d-2}, e_2\in \SS_0^{d-1}$, and satisfies that
  $Nw(\wh{E}_j)=k_0  \wt{w}(E_j)\in \NN$ for  each $j$.
This proves Assertion A for $d\ge 4$.

\section{Preliminary lemmas}

For the proof of Theorem \ref{sphdesign201}, in addition to the convex partitions of the weighted sphere (i.e.,  \thmref{thm-partition}),  we shall also need  several preliminary lemmas,  which we  state or prove in this section.

Throughout this section, $w$ denotes a normalized doubling weight on $\sph$ with doubling constant $L_w$. All the general constants $c_w, C_w, \da_w$ depend only on $d$ and the doubling constant of $w$.  Given $1\leq p\leq \infty$, we denote by $\|\cdot\|_{p,w}$ the $L^p$-norm defined with respect to the measure $w(x) d\s(x)$ on $\sph$.
A finite subset $\Lambda$ of $\sph$ is called $\va$-separated for a given $\va>0$ if $\dist(\o,\o')\geq \va$ for every two distinct points $\o, \o'\in\Ld$.
A $\va$-separated subset $\Ld$ is called maximal if  $\bigcup_{\o\in\Lambda}B(\o,\va)=\sph$.

The following weighted polynomial inequalities were established in  \cite{F_W} (see also \cite[Theorem 5.3.6.]{Dai2} and \cite[Theorem 5.3.4.]{Dai2}):
\begin{lem}\cite{F_W}\label{sphdesign301} Let $n$ be a  positive integer and  $\d$  a parameter in $ (0,1)$. Assume that $1\leq p<\infty$ and $\Ld$ is a  $\f\d n$-separated subset
of $\sph$ .
\begin{enumerate}[\rm (i)]
\item   For every spherical polynomial $f\in\Pi_n^d$,
\begin{equation}\label{}
\(\sum_{\o\in \Lambda} \left|\osc(f)(\o)\right|^p w(B_\og)\)^{1/p}\leq C_w \da\|f\|_{p,w},
\end{equation}
where $B_\og=B(\og, n^{-1}\da)$
and $\osc(f)(\og)=\max_{y\in B_\og} |f(y)|$  for $\og\in \Ld$.

\item If, in addition,    $\Ld$  is   maximal $\f {\da} n$-separated  with $0<\da<c_w$ and $c_w$ being  a small constant, then for each given parameter $\ga>1$, and every  $f\in \Pi_n^d$,
   \begin{align*}
   \|f\|_{p,w} &\sim \Bl( \sum_{\og\in \Ld} (\min_{x\in \ga B_\og} |f(x)|^p) w(B_\og)\Br)^{\f1p}
    \sim \Bl( \sum_{\og\in \Ld} (\max_{x\in \ga B_\og} |f(x)|^p)w(B_\og)\Br)^{\f1p},
   \end{align*}
   where  $B_\og=B(\og, n^{-1}\da)$  for $\og\in \Ld$, and the constants of equivalence depend only on $d$, $\ga$ and the doubling constant of $w$.
   \end{enumerate}
   \end{lem}

The   weighted
Christoffel function on $\sph$  is defined by
$$\l_n (w, x):=\inf_{P_n(x)=1} \int_{\sph}
 |P_n(x)|^2 w(x)\, d\s(x),\   \ n=0,1,2,\cdots,$$
 where  the infimum is taken over all spherical polynomials of degree $n$ on $\sph$ that take the value $1$ at the point $x\in\sph$.  The following pointwise estimate of  $\l_n (w, x)$ is proved in \cite{DaiFen}. (See also \cite{MT2} for the case of $d=2$).
 \begin{lem}\label{lem-3-2-0}\cite{DaiFen} For $x\in\sph$ and $n\in\NN$,
\begin{equation}\label{4-2}\l_n(w,x)\sim \int_{B(x, n^{-1})} w(y)\, d\s(y),\end{equation}
where the constant of equivalence depends only on $d$ and $L_w$.
\end{lem}

The tangential gradient of $f\in C^1(\RR^d)$ is defined as
$$\nabla_0 f(x) = \nabla f(x)-\f{x\cdot \nabla f(x)}{\|x\|^2} x=\nabla F(x),\   \  \   \  x\in \RR^d\setminus\{0\},$$
where $F(y)=F_x(y)= f\Bl( \f {\|x\| y}{\|y\|}\Br)$  for $y\in \RR^d\setminus\{0\}$.

We shall need the following lemma from \cite{Bond1}, whose  detailed proof can be found in the book \cite[p.145-147]{Dai2}.

\begin{lem}\label{cor-4-2} \cite{Bond1} For $\va>0$, $x\in\sph$ and
 each fixed spherical polynomial   $P\in\Pi_n^d$, the differential equation
\begin{equation}\label{4-7}
\begin{cases}
y'(s) &=\df{\nabla_0 P(y(s))}{\max\{\va,\|\nabla_0 P(y(s))\|\}},\   \   \  s\ge 0,\\
y(0)&=x\in\sph,\end{cases}\end{equation}
has a unique solution $y=y(P,s)\equiv y(P, x; s)\in \sph$ on $[0,\infty)$ with  the following  properties:
 \begin{enumerate}[\rm (i)]
 \item   $y(P,x; s) \in \sph$ for all  $s\ge 0$, $x\in\sph$ and $P\in\Pi_n^d$;
     \item for each fixed $s\ge 0$ and $x\in\sph$,
$P\mapsto y(P,x; s)$ is a continuous mapping from   $\Pi_n^d$ to $\sph$;
\item If $x, x'\in \sph$ and $x\neq x'$, then
    $$ y(P, x; s)\neq y(P, x'; s),\   \  \forall P\in \Pi_n^d, \   \  \forall s\ge 0. $$

\end{enumerate}
\end{lem}

As a consequence of Lemma \ref{cor-4-2}, we have the following useful corollary:

\begin{cor} \label{cor-4-3}Let $R\subset \sph$ be a closed geodesically convex subset of $\sph$ with $\max_{x,y\in R} \dist(x,y)<\f \pi2$. Assume that $P\in\Pi_n^d$ and $z_{\max}\in R$ is such that $P(z_{\max})=\max_{z\in R} P(z)$. Then for each $x\in R$,
$$ P(z_{\max})-P(x) \ge (\min_{z\in R} \|\nabla_0 P(z)\|)  \dist(x, \p (R)),$$
where $\dist(x, \p (R))=\inf_{y\in \p (R)} \dist(x,y)$.
\end{cor}

Corollary \ref{cor-4-3} was used without proof in \cite{Bond2}.   For completeness, we include a proof here.

\begin{proof}Without loss of generality, we may assume that $x$ is an interior point of $R$ and $\min_{z\in R} |\nabla_0 P(z)|=\va>0$.  By Lemma \ref{cor-4-2}, there exists a continuously differentiable function $y: [0,\infty)\to \sph$ satisfying the equation \eqref{4-7}.
We claim that there exists $s>0$ such that $y(s)\notin R$. Assuming otherwise, we then have that
$$ y'(s) =\df{\nabla_0 P(y(s))}{\|\nabla_0 P(y(s))\|},\   \   \  s\ge 0,$$
and hence,  for any $t>0$,
\begin{align*}
    P(y(t))-P(x)=\int_0^t \nabla_0 P(y(s))\cdot y'(s)\, ds =\int_0^t \|\nabla_0 P(y(s))\|\, ds \ge \va t.
\end{align*}
Letting $t\to \infty$, and taking into account the fact that $y(t)\in\sph$,  we conclude  that
the polynomial $P$ is not bounded on $\sph$, which contradicts the extreme value theorem. This proves the claim.

Now set
$$t_0=\sup\{ t\ge 0:\  \ y(s)\in R\   \ \forall s\in [0, t]\}.$$
 Since $y: [0,\infty)\to \sph$ is continuous and $x$ is an interior point of $R$,  it follows that  $0<t_0<\infty$,  $y(t_0)\in \p(R)\subset R$, and
 $$\dist(x, y(t_0))\leq \int_0^{t_0} \|y'(s)\|\, ds =t_0.$$
 This implies that
 \begin{align*}
    P(z_{\max})-P(x)&\ge P(y(t_0))-P(y(0)) =\int_0^{t_0} \|\nabla_0 P(y(s))\|\, ds \\
   & \ge (\min_{z\in R} |\nabla_0 P(z)|) t_0 \ge (\min_{z\in R} |\nabla_0 P(z)|)\dist(x, \p(R)).
 \end{align*}
\end{proof}

We will use the following lemma from algebraic topology.

\begin{lem}\cite[Ths. 1.2.6 and 1.2.9]{topology}\label{topology} Let $V$ be a finite dimensional real Hilbert space with inner product $\la\cdot,\cdot\ra$, and $\Omega$  a bounded open subset of $V$  which contains the zero vector in $V$.
 If $f: V\to V$  is a continuous mapping on the space $V$  with the property  $\la f(x),x\ra>0$ for all $x\in \p\Omega$, then there exists a vector  $x_0\in \Omega$ such that
$f(x_0)=0$.
\end{lem}

For convenience, we introduce the following definition:
\begin{defn}\label{def4-5:2017}For  $x, y\in \sph$ with  $\t:=\dist(x,y)\in (0, \f \pi2)$,  let $\ga_{[x,y]}: [0,1]\to \Arc (x,y)$ denote the   parametric representation of $\Arc(x,y)$given by
  \begin{equation}\label{3-1-0}
    \ga_{[x,y]}(t) =x\cos (\t t) +\xi \sin (\t t)
\end{equation}
 where $\xi\in\sph$ is such that $\xi\sin\t$ is the orthogonal projection of $y$ on the space $\{ y\in\RR^d:\ \ y\cdot x=0\}$; that is,
$\xi=\xi_{x,y}=\f {y-x\cos \t}{\sin \t}  $ .
In the case when $x=y$, we also set $\ga_{[x,y]}(t)=x$ for $t\in [0,1]$.
  \end{defn}

  A convex set $G$ in $\RR^d$ is said to be strictly convex if for any two distinct points $p, q\in G$ and any $t\in (0,1)$, $tp+(1-t)q$ is an interior point of $G$.

\begin{lem}\label{prop-3-3} \cite{Bond2} Let $R$ be a closed geodesically convex subset of $\sph$ with $\max_{p,q\in R} \dist(p,q)<\pi/2$. Assume that  $x_0$ is  a given  interior point of $R$ and let  $T_0:=\{ y\in \RR^d:\  \ x_0\cdot y=0\}$. Then the following statements hold:
\begin{enumerate}[\rm (i)]
\item For each $y\in T_0\setminus \{0\}$, there exists a unique point $x_y\in R$ such that $ x_y\cdot y =\max_{z\in R} z\cdot y$. Furthermore,  $y\longmapsto x_y$ is a continuous mapping on the set  $T_0\setminus\{0\}$ with the property that the function  $t\mapsto y\cdot (\ga_{[x_y,w]}(t))$ is decreasing on $[0,1]$
for each given    $y\in T_0\setminus\{0\}$ and $w\in R\setminus\{ x_y\}$.
\item Let $z_0$ be an arbitrary interior point of $R$.  Given $\da, \va\in (0,1/2)$, define
                \begin{equation} \label{4-5:2017} A(y)=\begin{cases}
                z_0, &\  \ \text{if $y=0$;}\\
                \ga_{[z_0, x_y]}\Bl( (1-\da)\min\{ 1, \f {|y|}\va\}\Br),&\  \  \text{if $y\in T_0\setminus \{0\}$}.\end{cases}
                \end{equation}
                Then $y\mapsto A(y)$ is a continuous mapping from  $T_0$ to $R$.
    \end{enumerate}
\end{lem}

Lemma \ref{prop-3-3}  was essentially  proved in \cite{Bond2}. However,
since the proof there  is rather sketchy, we  include a more  detailed proof of the lemma here.
\begin{proof} (i)
  Let $P_0(z)=z-(z\cdot x_0) x_0$ denote the orthogonal projection of  $z\in\sph$ onto the space $T_0$. Firstly, we  prove  that $D_0=P_0(R)$ is a strictly convex set in the space $T_0$.
To see this, let $S_{0}^{+}=\{ y\in\sph:\  \ y\cdot x_0>0\}$ and  $U_0:=\{y\in T_0:\  \ \|y\|<1\}$.   Clearly,  $R\subset S_{0}^{+}$, $D_0\subset U_0$, and  $P_0$ is a continuous mapping   from $S_0^{+}$ onto the set $U_0$ with  continuous inverse given by
 $P^{-1}(u)=u+\sqrt{1-\|u\|^2} x_0$ for $u\in U_0$.  Thus, to show $D_0$ is strictly convex, it suffices to prove that for any  two distinct points $u, v \in D_0$, and every $w=t u+(1-t)v$ with  $t\in (0,1)$, $P_0^{-1}(w)$ is an interior point of $R$.
Note first  that
\begin{equation}\label{4-4:2017}
    P_0^{-1} (w)=t P_0^{-1}(u)+(1-t) P_0^{-1}(v) +\a_0 x_0\in S_0^{+}
\end{equation}
with $\a_0=\sqrt{1-\|w\|^2} -t\sqrt{1-\|u\|^2}-(1-t)\sqrt{1-\|v\|^2}$.
We claim  that $\a_0>0$.  Indeed, since $g(x)=\|x\|^2$ is a strictly convex function on $\RR^d$, we have $\|w\|^2 < t\|u\|^2+(1-t)\|v\|^2$. Since   $\vi(s)=\sqrt{s}$ is a  concave function on $[0,\infty)$, this implies that
\begin{align*}
    \sqrt{1-\|w\|^2}&> \sqrt{1-t\|u\|^2-(1-t) \|v\|^2} =\vi\Bl(t(1-\|u\|^2)+(1-t)(1-\|v\|^2)\Br)\\
    &\ge t\vi(1-\|u\|^2)+(1-t)\vi(1-\|v\|^2),
\end{align*}
which in turn implies that $\al_0>0$.
Now setting
 $$p:=\f {t P_0^{-1}(u)+(1-t) P_0^{-1}(v) }{\|t P_0^{-1}(u)+(1-t) P_0^{-1}(v)\|},$$
  we have $p\in \Arc (P_0^{-1}(u), P_0^{-1}(v))\subset R$.   \eqref{4-4:2017} then implies that
 $P_0^{-1}(w)\in\Arc (p,x_0)\subset R$.
 To show that  $P_0^{-1}(w)$ is in fact  an interior point of $R$,  let $\da\in (0,1)$ be such that $\{ \f {x_0+\eta}{\|x_0+\eta\|}:\  \ \|\eta\|<\da\}\subset  R$.  Then for    $z=P_0^{-1}(w) +\al_0 \eta\in \sph$  with  $\eta\in\RR^d$ satisfying $\|\eta\|<\da$, we use \eqref{4-4:2017} to obtain
 \begin{align*}
    z=t P_0^{-1}(u)+(1-t) P_0^{-1}(v) +\al_0 (x_0+\eta)\in\Arc\Bl(p, \f {x_0+\eta}{\|x_0+\eta\|}\Br)\subset R.
 \end{align*}
 This shows that $P_0^{-1}(w)$ is an interior point of $R$, and hence proves that $D_0$ is a strictly convex subset of $T_0$.

Secondly, we show that for each $y\in T_0\setminus\{0\}$, there exists a unique $x_y \in R$ such that
$x_y \cdot y=\max_{z\in R} z\cdot y$. Indeed, this follows directly from the facts that $D_0$ is strictly convex and $\max_{z\in R} z\cdot y=\max_{z\in R} P_0(z)\cdot y$.

Thirdly, we show that $y\mapsto x_y$ is continuous on $T_0\setminus\{0\}$.
Let $y, z$ be two distinct nonzero vectors in $T_0$, and let $p=P_0(x_y)$ and $q=P_0 (x_z)$.
Then $p, q\in D_0$,
$p\cdot y =x_y\cdot y=\max_{u\in D_0} u\cdot y$,
 and $q\cdot z =x_z\cdot z=\max_{u\in D_0} u\cdot z$.
  Since $\max_{x, y\in R} \dist(x, y)<\f \pi 2$, we also have that $\|x_y-x_z\|\leq C_R \|p-q\|$. For convenience, we set  $H_{p, y}:=\{ u\in T_0:\  \ u\cdot y=p\cdot y\}$,  $H_{p, y}^-:=\{ u\in T_0:\  \ u\cdot y\leq p\cdot y\}$ and $H_{p, z}^+:=\{ u\in T_0:\  \ u\cdot z\ge p\cdot z\}$.
Clearly,  $q\in H_{p, z}^{+}\cap D_0$ and $D_0\subset H_{p, y}^{-}$. Note also that
$$\rho(\f {p+q}2, H_{p,y}):=\min_{u\in H_{p,y}} \|\f {p+q}2-u\|=\f 1{2\|y\|} (p-q)\cdot y.$$
Since $p\cdot z\leq q\cdot z$, it follows that
\begin{equation}\label{4-7:2017}
 \rho(\f {p+q}2, H_{p,y})\leq \f 1{2\|y\|} (p-q)\cdot (y-z)\leq \f 1{\|y\|} \|y-z\|.
\end{equation}
On the other hand, since $D_0$ is strictly convex,
$\rho(\f {p+u} 2, H_{p,y})>0$ for any $u\in D_0\setminus\{p\}$. Thus, given any $\va>0$,
$$\min\Bl\{ \rho(\f {p+u} 2, H_{p,y}):\  \  u\in D_0,\  \ \|u-p\|\ge \va\Br\}=\da>0.$$
In particular,   if $\|x_y-x_z\|\ge \va/C_R$, then  $\|p-q\|\ge \va$ and hence  by \eqref{4-7:2017},  $ 0<\da\|y\| \leq  \|y-z\|$. This shows the  continuity of the mapping $y\mapsto x_y$.

Finally, we show that  given  each   $y\in T_0\setminus\{0\}$ and $w\in R\setminus\{ x_y\}$,  the function  $t\mapsto y\cdot (\ga_{[x_y,w]}(t))$ is decreasing on $[0,1]$.  Let $G$ denote the great circle passing through $x_y$ and $w$, and  $\eta$  the orthogonal projection of $y$ onto the plane spanned by the vectors $x_y$ and $w$. Let $\xi=\f {\eta}{\|\eta\|}\in G$. Then
$$ y\cdot (\ga_{[x_y,w]}(t))=\|\eta\| \xi\cdot (\ga_{[x_y,w]}(t)),\   \  \forall t\in [0,1].$$
 Since $w\cdot \xi\leq x_y\cdot \xi$, it suffices to show that  the arc $\Arc(x_y,  w)$ lies between the  points $\xi$ and $-\xi$ on the great circle $G$.
 Indeed, since $\max_{u,v\in R}\dist(u,v)<\f \pi2$, we have $x_y\cdot \xi =\|\eta\|^{-1} x_y\cdot y>0$.  Hence, $\xi$ cannot be in the interior of the geodesic arc $\Arc(w, x_y)$ since otherwise $\xi\in R$ and $\|\eta\|=\xi\cdot \eta =\xi\cdot y >x_y\cdot y$. Similarly, one can also show that  that $-\xi$ can not lie in the interior of $\Arc(w, x_y)$. Indeed, assuming otherwise, we have that  $-\xi\in R$, which would  imply that $x_y \cdot y =\|\eta\|\xi\cdot x_y <0$, yielding a contradiction.

(ii)  Let  $\da_1 \in (0, \f12)$ be  such that  $B(z_0, \da_1)\subset R$. Since $x_y$ is on the boundary of $R$, it follows by (i) that $\ta(y)=\dist(x_y, z_0)\in [\da_1, \f \pi2)$ for each $y\in T_0\setminus\{0\}$, and is continuous in $y\in T_0\setminus\{0\}$. Furthermore, according to \eqref{3-1-0},
 for $y\in T_0\setminus\{0\}$,
 \begin{equation}\label{3-2}
    \ga_{[z_0, x_y]}(t) =z_0\cos \bl(t\t(y)\br) +\xi_y \sin \bl(t \t(y)\br),\   \  t\in [0,1],
 \end{equation}
 where
 $$ \xi_y := \f {x_y -z_0\cos \bl(\t(y)\br)}{\sin \bl(\t(y)\br)}.$$
 Since $0<\da_1\leq \t(y)<\f \pi2$, $y\mapsto \xi_y$ is a continuous mapping from $T_0\setminus\{0\}$ to $\sph$.
 Since the function  $h_\va(y):=\min\{1, \f{\|y\|}{\va}\}$ is continuous  on $T_0$,  it follows by   \eqref{3-2} that
 $ A(y)= \ga_{[z_0, x_y]}\Bl((1-\da)h_\va(y)\Br)$
 is continuous on $T_0\setminus\{0\}$.
 On the other hand, setting $t_y:=(1-\da) h_\va(y)\ta(y)$ for $y\in T_0\setminus \{0\}$, we have that
 $\lim_{\sub{y\to 0\\
 y\in T_0}}t_y= 0,$
 and hence
 \begin{align*}
    \lim_{\sub{y\to 0\\
    y\in T_0}} A(y) =\lim_{\sub{y\to 0\\
    y\in T_0}} z_0\cos (t_y) +\xi_y \sin (t_y)=z_0=A(0).
 \end{align*}
 This shows that $A(y)$ is continuous at $y=0$ as well.
\end{proof}

\section{Well-Separated Chebyshev-type cubature formulas}

This section is devote to the proofs of Theorem \ref{sphdesign201} and Corollary \ref{cor-1-2}.

Let $\Pi_{n,0,w}^d$ denote the set of all spherical polynomials $P$  of degree at most $n$ on $\sph$ with $\int_{\sph}P(x) w(x)\, d\s(x)=0$. Then   $\Pi_{n,0,w}^d$ is a finite dimensional real  Hilbert space equipped  with the inner product of the space $L^2(\sph; w(x)d\s(x))$.
 Let $G_{n,w}(\cdot, \cdot)$ denote the reproducing kernel of the Hilbert space $\Pi_{n,0,w}^d$.
 Clearly,  \eqref{1-5:2017} is equivalent to the following
\begin{equation}\label{5-2:2017}
     \sum_{j=1}^N G_{n,w}(z_j, z)=0,\   \  \forall z\in\sph.
\end{equation}
  Finally,
we recall that  $\|f\|_{1,w}=\int_{\sph} |f(x)|w(x)\, d\s(x)$.

The proof of Theorem \ref{sphdesign201} follows closely the methods used in  \cite{Bond2}. Indeed,  applying  \lemref{topology} to the  vector valued function   $F(P; z): =\sum_{j=1}^N G_{n,w}( \mathbf{x}_j (P), z)$,  $P\in \Pi_{n,0,w}^d$, $z\in\sph$ and  the open subset  $\Og:=\{ P\in \Pi_{n,0,w}^d:\   \ \|\nabla_0 P\|_{1,w}<1\}$ of the Hilbert space  $\Pi_{n,0,w}^d$,  and taking into account \eqref{5-2:2017}, we  reduce to showing

\begin{prop}\label{prop-5-2}   For each integer  $N\ge K_{w} M_{n,w}$, there exists a set $\{\mathbf{x}_\al\}_{\al\in \Ld}$ of $N$ continuous functions $P\mapsto \mathbf{x}_\al(P)$ from the space $\Pi_{n,0,w}^d$ to $\sph$ such that $\mathbf{x}_\al (P)\neq \mathbf{x}_{\al'}(P)$ for  $P\in \Pi_{n,0,w}^d$, $\al, \al'\in\Ld$ and $\al\neq \al'$, and such that   \begin{equation}\label{sphdesign402}
     \sum_{\al\in \Ld}  P(\mathbf{x}_{\al}(P))>0 \   \   \ \text{whenever $P\in \Pi_{n,0,w}^d$  and    $\|\nabla_0 P\|_{1,w}=1$},
  \end{equation}
  where $\Ld$ is an index set with cardinality $N$.
  If, in addition, $w\in L^\infty(\sph)$, then  the set of points $\{\mathbf{x}_\al(P)\}_{\al\in\Ld}$    is $c_w N^{-\f 1{d-1}}$-separated for every $P\in \Pi_{n,0,w}^d$; that is,
  \begin{equation}\label{sphdesign401}
\min_{\al, \al'\in \Ld, \al\neq \al'}\dist(x_{\al}(P), x_{\al'}(P))\ge c_w N^{-\f 1{d-1}},\   \  \forall P\in \Pi_{n,0,w}^d
\end{equation}
   for some positive constant $c_w$ depending only on $\|w\|_\infty$ and the doubling constant of $w$.

 \end{prop}

\subsection{Proof of Proposition \ref{prop-5-2}}
Throughout the proof, $K_w$ denotes a sufficiently large constant depending only on the doubling constant of $w$.  Set $\da=2^{\f1{s_w}}K_w^{-\f 1{s_w}}\in (0,1)$.
By \eqref{1-5-2017}, for any $x\in\sph$,
\begin{equation}\label{}
  w(B(x, \f \da n)) \ge \f 12 \da^{s_w} w(B(x, \f 1n))\ge \f{ \da^{s_w} }{2M_{n,w}}\ge   \f {\da^{s_w} K_{w}}{2N}\ge \f 1N.
\end{equation}

 Let $N, n\in \NN$ be such that   $N^{-\f 1{s_w}}=\f\da n$, where $\da\in (0,1)$ is a small constant depending only on the doubling weight of $w$. According to  Theorem \ref{thm-partition}, there exists  a   convex partition
 $\mathcal{R}=\{R_1, \cdots, R_M\}$ of $\sph$  with the properties that  for each $1\leq j\leq M$,  there exist a positive integer $k_j$ and a point $x_j\in R_j$ such that $w(R_j)=\f {k_j}N$ and $ B(x_j, \f{c_d \da}n )\subset R_j \subset B(x_j, \f {c_d'\da} n)$.

For each $1\leq j\leq M$, set
$r_j = \f {\da^2} n  \bl(N w(R_j))\br)^{-\f 1{d-1}}.$
 Since $k_j r_j^{d-1}=\da^{d-1} \Bl(\f \da n\Br)^{d-1}$,
there exists a   set of $k_j$  points
$x_{j, 1}, \cdots, x_{j, k_j}$ in the set $ R_j$ which  are $2r_j$-separated and satisfy  $x_{j,1}=x_j$,   and $B(x_{j,i}, r_j)\subset R_j$ for all $1\leq i \leq k_j$.
If,  in addition,  $w\in L^\infty$, then
$$ w(R_j) \leq w(B(x_j, \f {c_d'\da} n))\leq  C_d\|w\|_\infty |B(x_j, \f \da n)|\leq  C_d\|w\|_\infty \Bl( \f \da n\Br)^{d-1},$$
which implies
\begin{equation}\label{5-6-2017}
   r_j \ge C_d \|w\|_\infty^{-\f 1{d-1}}\f {\da^2} n  N^{-\f 1{d-1}}  \Bl( \f \da n\Br)^{-1}=C_d \|w\|_\infty^{-\f 1{d-1}}\da N^{-\f 1{d-1}}.
\end{equation}

Let $P\in\Pi_{n,w,0}^d$.  It is easily seen that  $x\cdot \nabla_0 P(x)=0$ for all $x\in\sph$. Thus,
   according to \lemref{prop-3-3}, if $\nabla_0P(x_j)\neq 0$, there exists a unique point $z_{j,P}\in R_j$  such that   $ z_{j,P}\cdot\Bl(\nabla_0 P(x_j)\Br) =\max_{z\in R_j}  z\cdot \Bl(\nabla_0 P(x_j)\Br)$.
  Now for each $P\in\Pi_{n,0,w}^d$, we define
\begin{align}\label{sphdesign400}
    \mathbf{x}_{j,i}(P)&=\begin{cases} x_{j,i}&\  \ \text{if $\nabla_0 P(x_j)=0$;}\\
    \ga_{[x_{j,i}, z_{j, P}]}\Bl( (1-\da)\min\bl\{1,  \f{|\nabla_0 P(x_j)|}\va\br\}\Br),&\  \ \text{if $\nabla_0 P(x_j)\neq 0$.}
    \end{cases}
\end{align}
We claim that for each $1\leq j\leq M$ and $1\leq i\leq k_j$, $P\mapsto \mathbf{x}_{j,i}(P)$ is a continuous function from $\Pi_{n,w,0}^d$ to $\sph$.  Indeed, by \lemref{prop-3-3}, we may write  $\mathbf{x}_{j,i} (P)= A(\nabla_0 P(x_j))$, where $A$ is defined in \eqref{4-5:2017} with $x_j, x_{j,i}, \nabla_0 P(x_j)$ in place of $x_0, z_0$ and $y$ respectively.  According to \lemref{prop-3-3} (ii),   $\mathbf{x}_{j,i}(P)$ is a continuous function of $\nabla_0 P(x_j)$. On the other hand, however, since $\Pi_{n,w,0}^d$ is a finite dimensional vector space, the mapping $P\mapsto \nabla_0 (P)(x_j)$ is continuous on $\Pi_{n,0,w}^d$.  This proves the claim.

Now we set  $\Ld:=\{ (j,i):\  \ 1\leq j\leq M,\   \  1\leq i\leq k_j\}$, and    turn to the proof of \eqref{sphdesign402}.
  Assume that   $P\in \Pi_{n,0,w}^d$ and   $\|\nabla_0 P\|_{1,w}=1$.  Let $z_{j,P}$, $x_{j,i}$  be  defined as above.
  For convenience, we also set
  $$ z_{j,i,P} :=\begin{cases} z_{j,P}, &\  \ \text{if $\nabla_0 P(x_j)\neq 0$}\\
  x_{j,i}, &\  \ \text{if $\nabla_0 P(x_j)=0$}\end{cases}\   \   \text{for  $1\leq i\leq k_j$ and $1\leq j\leq M$}.$$
 Let $\delta\in (0,1)$ be a parameter to be specified later. Define
 $$ y_{j,i} :=\begin{cases} \g_{[x_{j,i},z_{j,P}]}(1-\delta), &\  \ \text{if $\nabla_0 P(x_j)\neq 0$}\\
  x_{j,i}, &\  \ \text{if $\nabla_0 P(x_j)=0$},\end{cases}\   \ \text{for  $1\leq i\leq k_j$ and $1\leq j\leq M$}.$$
 Let $z_{j,\max}\in R_j$ be such that $P(z_{j, \max})=\max_{z\in R_j} P(z)$.
We then  split the sum $\f1N \sum_{j=1}^M \sum_{i=1}^{k_j} P\bl(\mathbf{x}_{j,i}(P)\br)$ on  the left hand side of \eqref{sphdesign402} into the following  four parts:
\begin{align}
    &\f1N \sum_{j=1}^M k_j P\bl(z_{j,\max}\br)-\f1N \sum_{j=1}^M\sum_{i=1}^{k_j}
    \Bl[P\bl(z_{j,\max}\br)-P\bl(z_{j,i,P}\br) \Br]\notag\\
    &-\f1N \sum_{j=1}^M\sum_{i=1}^{k_j}
    \Bl[P\bl(z_{j,i,P}\br)-P\bl(y_{j,i}\br) \Br]-\f1N \sum_{j=1}^M \sum_{i=1}^{k_j} \Bl[P\bl(y_{j,i}\br)-P\bl(\mathbf{x_{j,i}}(P)\br) \Br]\notag\\
    &=\Sigma_1-\Sigma_2-\Sigma_3-\Sigma_4.\label{5-8-2017}
\end{align}

 Firstly, we estimate the first sum $\Sigma_1$ from below.   Use   \lemref{cor-4-3}, we  obtain
\begin{align}
  \Sigma_1&= \sum_{j=1}^M \int_{R_j} \Bl[P(z_{j, \max})-P(x)\Br]\, w(x)\, d\s(x)\notag\\
 &\ge  \sum_{j=1}^M \Bl(\min_{z\in R_j}|\nabla_0 P(z)|\Br) \int_{B(x_j, \f{c_d  \da}{2 n})}  \dist(x, \p R_j) \, w(x)\, d\s(x)\notag\\
 &\ge  C_w\f {\da} {n}   \sum_{j=1}^M w(R_j) \Bl(\min_{z\in R_j}|\nabla_0 P(z)|\Br) \ge c_w \f \da n \|\nabla_0 P\|_{1,w},\label{Sum1}
\end{align}
where we used  the fact that $B(x_j,\f{c_d \d}n)\subset R_j\subset B(x_j, \f {c_d'\d} n)$ and the doubling property of $w$ in the third step, and \lemref{sphdesign301}(ii) in the last step.

Secondly, we prove the following upper estimate of the second sum $\Sigma_2$:
\begin{equation}\label{5-9:2017}
   \Sigma_2:= \f 1N  \sum_{j=1}^M \sum_{i=1}^{k_j}\Bl[P(z_{j, \max})-P(z_{j,i, P})\Br]\leq \f{C_w\d^2}n.
\end{equation}
For simplicity, we write   $\ga_{j,i}(t)=\ga_{[z_{j,i, P}, z_{j,\max}]}(t)$.
 If  $\nabla_0 P(x_j)\neq 0$, then $z_{j,i,P}=z_{j,P}$, and hence,  by Lemma \ref{prop-3-3} (i), the function $\nabla_0 P(x_j)\cdot \ga_{j,i}(t)$ is decreasing on $[0,1]$, namely,
   \begin{equation}\label{0-2}\nabla_0 P(x_j) \cdot \ga_{j,i}'(t)\leq 0,\   \ \forall t\in [0,1].\end{equation}
 On the other hand, note that  if $\nabla_0 P(x_j)= 0$, then $z_{j,i,P}=x_{j,i}$ and then the  inequality \eqref{0-2}  holds trivially. Thus,
\begin{align*}
    &P(z_{j,\max})-P(z_{j,i, P})=\int_0^1 \nabla_0 P(\ga_{j,i}(t))\cdot \ga_{j,i}'(t)\, dt\notag \\
    &\leq
    \int_0^1 \Bl(\nabla_0 P(\ga_{j,i}(t))-\nabla_0 P(x_j)\Br)\cdot \ga_{j,i}'(t)\, dt
    \leq \dist(z_{j,\max}, z_{j,i, P}) \max_{z\in R_j} \|\nabla_0 P(z)-\nabla_0 P(x_j)\|\notag\\
    &\leq \f {c_d\da} n \max_{z\in R_j} \|\nabla_0 P(z)-\nabla_0 P(x_j)\|.
\end{align*}
It follows that
\begin{align*}
  \Sigma_2&\leq \f{c_d\da} n  \sum_{j=1}^M w(R_j) \max_{z\in R_j} \|\nabla_0 P(z)-\nabla_0 P(x_j)\|
  \leq  \f{C_w\da^2} n\|\nabla_0 P\|_{1,w} = \f{C_w\da^2} n,
  \end{align*}
  where the second step uses  \lemref{sphdesign301} (i). This proves the estimate \eqref{5-9:2017}.

Thirdly, we show the following upper estimate of $\Sigma_3$:
\begin{equation}\label{5-7}
   \Sigma_3:= \f 1N \sum_{j=1}^M \sum_{i=1}^{k_j}  \Bl|P(z_{j,i, P})-P(y_{j,i})\Br|\leq   \f{C_w\d^2}n,
\end{equation}
For simplicity, we set $\al_{j,i}(t)=\ga_{[x_{j,i}, z_{j,i, P}]}(t)$.
Recall that $y_{j,i}=\al_{j,i}(1-\da)$.
 We  have
\begin{align*}
    |P(z_{j,i, P})-P(y_{j,i})|&=\Bl|\int_{1-\da}^1 \nabla_0 P(\al_{j,i}(t))\cdot \al_{j,i}'(t)\, dt\Br|\\
    &\leq \da \dist(z_{j,i, P}, x_{j,i}) \max_{z\in R_j}\|\nabla_0 P(z)\|\leq \f {C_d\da^2} n \max_{z\in R_j}|\nabla_0 P(z)|.
\end{align*}
It follows  by \lemref{sphdesign301} that \begin{align*}
   \Sigma_3&\leq \f {C_d\da^2} n  \sum_{j=1}^M w(R_j)  \max_{z\in R_j}\|\nabla_0 P(z)\|
   \leq C_w \f {\da^2} n\|\nabla_0 P\|_{1,w}=\f{C_w\d^2}n.
\end{align*}
This proves \eqref{5-7}.

Fourthly, we estimate above  the last sum $\Sigma_4$. Fix temporarily $1\leq j\leq M$ and a small parameter $\va\in (0,1)$. We consider the following three cases. If $\|\nabla_0 P(x_j)\|=0$, then $x_{j,i}=\mathbf{x}_{j,i}(P)=z_{j,i,P}=y_{j,i}$, and hence,
$P(\mathbf{x}_{j,i}( P))-P(y_{j,i})=0$.
If $\|\nabla_0 P(x_j)\|\ge \va>0$, then
$$\mathbf{x}_{j,i}(P)=\al_{j,i} \Bl( (1-\da) \min\Bl\{ \f{\|\nabla_0 P(x_j)\|}\va, 1\Br\}\Br)=\al_{j,i}(1-\da)=y_{j,i},$$
and hence in this case we also have  $P(\mathbf{x}_{j,i}( P))-P(y_{j,i})=0$.
Finally, if $0<\|\nabla_0P(x_j)\|<\va$, then
\begin{align*}
   \Bl|P(\mathbf{x}_{j,i}( P))-P(y_{j,i})\Br|&\leq \Bl|\int_{(1-\da)\f {\|\nabla_0 P(x_j)\|}\va}^{1-\da} \nabla_0 P(\al_{j,i}(t))\cdot \al_{j,i}'(t)\, dt\Br|\\
   &\leq \int_0^1 \bl\|\nabla_0 P(\al_{j,i}(t))-\nabla_0 P(x_j)\br\|\| \al_{j,i}'(t)\|\, dt+ \va \int_0^1 \|\al_{j, i}'(t)\|\, dt\\
   &\leq \f {C_d\da} n \max_{y, z\in R_j}\|\nabla_0 P(y)-\nabla_0 P(z)\|+\f {C_d\va \da}n.
\end{align*}
Putting the above together, we obtain
\begin{align}
    \Sigma_4&\leq \f 1N \sum_{j:\ \|\nabla_0P(x_j)\|<\va} \sum_{i=1}^{k_j} \Bl|P(x_{j,i}( P))-P(y_{j,i})\Br|\notag\\
   &\leq \f {C_d\da} n  \sum_{j=1}^M w(R_j)\max_{y, z\in R_j}\|\nabla_0 P(y)-\nabla_0 P(z)\|+ \f {C_d\va \da}n\notag\\
   &\leq C_w \f {\da^2} n\|\nabla_0 P\|_{1,w} +\f {C_d\va \da}n.\label{5-12:2017}
\end{align}

 Finally,  combining the estimates \eqref{Sum1}, \eqref{5-9:2017}, \eqref{5-7}, \eqref{5-12:2017} with \eqref{5-8-2017}, we obtain
\begin{align*}
    \f1N \sum_{j=1}^M \sum_{i=1}^{k_j} P\bl(\mathbf{x}_{j,i}(P)\br)&=\Sigma_1-\Sigma_2-\Sigma_3-\Sigma_4\\
    &\ge \f {c_w \da}n -\f{C_w \da^2}{n} -\f{C_w \da \va}n=\f \da n \Bl( c_w -C_w \da -C_w  \va\Br).
\end{align*}
To conclude the proof of \eqref{sphdesign402},  we just need to choose the parameters $\va,\da$ small enough so that $0<\va, \da<\f {c_w}{4C_w}$.

 It remains to show  the separation property \eqref{sphdesign401} under  the additional condition $w\in L^\infty(\sph)$. To this end,
 we need the following simple lemma.
\begin{lem}\label{lem-5-3:2017} Let  $z, \xi_1, \xi_2\in\sph$ be such that   $\xi_1\cdot z=\xi_2\cdot z=0$.
 Given $\t_1, \t_2\in (0,\f \pi 4]$, define
 $$\ga_i(t) =z\cos (\t_i t) +\xi_i \sin (\t_i t)=\ga_{[z, \eta_i]}(t),\   \  i=1,2,$$
 where $\eta_i=z\cos\t_i +\xi_i \sin\t_i$.
  Then for any $t\in (0,1)$,
$$ \|\ga_1(t)-\ga_2(t)\|\sim t \|\ga_1(1)-\ga_2(1)\|=t\|\eta_1-\eta_2\|.$$
\end{lem}

For the moment, we take Lemma \ref{lem-5-3:2017}  for granted and  proceed with the proof of \eqref{sphdesign401}. Without loss of generality, we may assume that $\|w\|_\infty=1$.  Set $\Ld=\{ (j,i):\  \ 1\leq j\leq M,\  \ 1\leq i\leq k_j\}$.
It is enough to prove that for $(j, i)\neq  (j',i')\in\Ld$ and every $P\in\Pi_{n,0,w}^d$ with $\|\nabla_0 P\|_{1,2}\leq 1$,
\begin{equation}\label{separation}
  \dist(\mathbf{x}_{j,i}(P),\mathbf{x}_{j',i'}(P)) \ge c_w N^{-\f 1{d-1}}.
\end{equation}

We first prove \eqref{separation} for the case of $j=j'$ and $1\leq i\neq i'\leq k_j$.  In this case, if $\nabla_0 P(x_j)= 0$, then $\mathbf{x}_{j,i}(P)=x_{j,i}$, $\mathbf{x}_{j,i'}(P)=x_{j,i'}$, and hence by \eqref{5-6-2017},  $\dist(\mathbf{x}_{j,i}(P), \mathbf{x}_{j,i'}(P))\ge r_j \ge C N^{-\f 1{d-1}}$.  If $\nabla_0 P(x_j)\neq  0$, then
$$ \mathbf{x}_{j,i}(P) =\ga_{[x_{j,i}, z_{j,P}]} \bl( (1-\da)\min\bl\{  \va^{-1}|\nabla_0 P(x_j)|, 1\br\} \br)=\ga_{ [z_{j,P}, x_{j,i}]} (\al),$$
with
$$\al=1- (1-\da)\min\Bl\{ \f {|\nabla_0 P(x_j)|}\va, 1\Br\}  \ge \da,$$
which using Lemma \ref{lem-5-3:2017} implies
\begin{align*}
    \dist\Bl(\mathbf{x}_{j,i}(P), \mathbf{x}_{j,i'}(P)\Br)&\ge  \Bl\|\ga_{ [z_{j,P}, x_{j,i}]} (\al)-\ga_{ [z_{j,P}, x_{j,i'}]} (\al)\Br\|\\
    &\ge c\al |x_{j,i}-x_{j,i'}|
    \ge c \da r_j \ge c\da N^{-\f 1{d-1}}.
\end{align*}

Next, we  prove \eqref{separation} for the case of  $1\leq j\neq j'\leq M$. Assume that   $1\leq i\leq k_j$ and $1\leq i'\leq k_{j'}$. To prove \eqref{separation},  it suffices to show that $\dist(\mathbf{x}_{j,i}(P), \p R_j)\ge c_w r_j\ge  c N^{-\f 1{d-1}}$.
Without loss of generality, we may assume that $\|\nabla_0 P(x_j)\|>0$ since otherwise $\mathbf{x}_{j,i}(P)=x_{j,i}$ and the claim  is obvious.
We first recall that $B(x_{j,i}, r_j)\subset R_j$, and $$\mathbf{x}_{j,i}(P)=\ga_{[x_{j,i}, z_{j, P}]}(t_{j}),$$
with $t_{j}= (1-\da)\min \Bl\{ \f{\|\nabla_0 P(x_j)\|}\va, 1\Br\}.$
Hence, setting $\ta_{j,i} =\dist(x_{j,i}, z_{j,P})$, we have
$$\dist(\mathbf{x}_{j,i}(P), x_{j,i}) =t_j\dist(x_{j,i}, z_{j,P})\leq (1-\da)\t_{j,i}.$$

Let $h: R_j \to T_{x_{j,i}}$, $z\mapsto z-(z\cdot x_{j,i}) x_{j,i}$ denote the orthogonal projection onto the tangential  space $T_{x_{j,i}}:=\{ y\in\RR^d:\ \  y\cdot x_{j,i}=0\}$. According to the proof of Lemma \ref{prop-3-3},   the set  $S =h(R_j)$ is strictly convex in the space  $T_{x_{j,i}}$, $h(x_{j,i})=0$, and $\|h(z)\|=\sin (\dist(x_{j,i}, z))$ for $z\in R_j$. Furthermore,  $$ \dist(x_{j,i}, \p R_j)\ge d(0, \p S)\ge \sin (r_j),\  \ \ \dist(\mathbf{x}_{j,i}(P), \p R_j)\ge  d(h(\mathbf{x}_{j,i}(P)), \p S).$$
Write
$$\ga_{[x_{j,i}, z_{j,P}]}(t)=x_{j,i} \cos (\ta_{j,i} t) +\xi_{j,i} \sin (\ta_{j,i} t),\   \ t\in [0,1],$$
where
$$\xi_{j,i}=\f {h(z_{j,P})}{\|h(z_{j,P})\|}=\f {h(z_{j,P})}{\sin (\t_{j,i})}\in\sph\cap T_{x_{j,i}}.$$
Then,
\begin{align*}
    h\Bl(\ga_{[x_{j,i}, z_{j,P}]}(t)\Br)=\xi_{j,i} \sin (\ta_{j,i} t)=\f {h(z_{j,P})}{\sin (\t_{j,i})}\sin (\ta_{j,i} t) ,\   \ t\in [0,1].
\end{align*}
This implies that
\begin{align*}
  h(\mathbf{x}_{j,i}(P)) =  \f {h(z_{j,P})}{\sin (\t_{j,i})}\sin (\ta_{j,i} t_j)=(1-s) h(x_{j,i}) +s  h(z_{j,P}),
\end{align*}
where
$$0\leq  s=\f{\sin (\ta_{j,i} t_j)}{\sin \t_{j,i}}\leq t_j\leq 1-\da.$$
Here we used the fact that the function $\f{\sin t}t$ is decreasing on $[0,\f \pi2]$.
Since the function $d(u, \p S)$ is concave on $S$, it follows that
\begin{align*}
    \dist(\mathbf{x}_{j,i}(P), \p R_j)&\ge \dist (h(\mathbf{x}_{j,i}(P), \p S) \ge (1-s) d (h(x_{j,i}), \p S)+s\cdot d(h(z_{j,P}),\p S)\\
    &\ge (1-s) d(0, \p S) \ge \da \sin (r_j)\ge C r_j \ge C N^{-\f 1{d-1}}.
\end{align*}

\subsection{Proof of Lemma \ref{lem-5-3:2017}}

 For $t\in (0,1)$ and $\t_1, \t_2 \in (0, \f \pi 4)$,
\begin{align*}
    \|\ga_1(t)-\ga_2(t)\|^2 =&|\cos (\t_1 t)-\cos(\t_2 t)|^2 +|\xi_1 \sin (\t_1 t) -\xi_2 \sin (\t_2 t)|^2 \\
    =& 4 \sin^2 \Bl(\f {\t_1-\t_2}2 t \Br) \sin^2 \Bl(\f {\t_1+\t_2}2 t \Br)+ |\sin (\t_1 t)-\sin (\t_2 t)|^2 \\
    &+ 2\sin (\t_1 t) \sin (\t_2 t) (1-\xi_1\cdot \xi_2)\\
    =&4 \sin^2 \Bl(\f {\t_1-\t_2}2 t \Br) \sin^2 \Bl(\f {\t_1+\t_2}2 t \Br)+4 \sin^2 \Bl(\f {\t_1-\t_2}2 t \Br) \cos^2 \Bl(\f {\t_1+\t_2}2 t \Br)\\
   & + \sin (\t_1 t) \sin (\t_2 t) \|\xi_1-\xi_2\|^2\\
    =&4 \sin^2 \Bl(\f {\t_1-\t_2}2 t \Br) + \sin (\t_1 t) \sin (\t_2 t) \|\xi_1-\xi_2\|^2.
\end{align*}
Thus,
\begin{align*}
   \|\ga_1(t)-\ga_2(t)\|&\sim \Bl(|\t_1-\t_2|   +  \sqrt{\t_1 \t_2} \|\xi_1-\xi_2\|\Br)t\\
   &\sim t \|\ga_1(1)-\ga_2(1)\|.
\end{align*}

\subsection{Proof of  Corollary \ref{cor-1-2}} The upper estimate,
$$ \mathcal{N}_n (wd\s_d) \leq C \max_{x\in \sph}\f 1 {w(B(x, n^{-1}))}, $$
follows directly from Theorem \ref{sphdesign201}. It remains to show the matching lower estimate.

  Assume that   $\Ld=\{\og_j\}_{j=1}^N$ is a  set of $N$ distinct nodes on $\sph$ such that
 \begin{equation}\label{5-14:2017}\int_{\sph} f(x) w(x)\, d\s_d(x)=\f 1N \sum_{j=1}^N f(\og_j),\   \   \ \forall f\in\Pi_n^d.\end{equation}

 According to \cite[Lemma 4.6]{F_W}, given a positive integer $\ell\ge s_w +d+1$,  there exists a nonnegative algebraic  polynomial $P_n$ of degree at most $n/4$ on $[-1,1]$ such that
 \begin{equation}\label{5-15:2017}
 0\leq   P_n (\la x,  y\ra) \sim n^{d-1} (1+n \dist(x, y))^{-2\ell},\   \   \  \forall x, y\in \sph,
 \end{equation}
 where $\la x, y\ra$ denotes the dot product of $x, y\in\RR^d$.
Using \eqref{5-14:2017} and \cite[Theorem 4.2]{F_W}  we have that for $p=1, 2$,
 \begin{equation*}
   \|P_n(\la x, \cdot\ra )\|_{p, w} \sim \Bl( \f 1N \sum_{j=1}^N |P_n(x\cdot \og_j)|^p\Br)^{1/p},\   \  \  \forall x\in\sph,
 \end{equation*}
 where $\|\cdot\|_{p,w}$ denotes the Lebesgue $L^p$-norm defined with respect to the measure $w(x)\, d\s_d(x)$ on $\sph$.
   Since the norm  $\|\cdot\|_{\ell^p}$ is a decreasing function in $p>0$,   it follows that
 \begin{align}\label{5-16:2017}
   \|P_n(\la x, \cdot\ra) \|_{2,w}\leq C \sqrt{N} \|P_n(\la x, \cdot\ra)\|_{1,w},\   \  \forall x\in\sph.
 \end{align}

 Next,   by \cite[Corollary 3.4]{F_W},
 $ \|P_n(\la x, \cdot\ra)\|_{p,w} \sim \|P_n(\la x, \cdot\ra)\|_{p, w_n}$ for $1\leq p<\infty$, where $w_n(x) =n^{d-1}\int_{B(x, n^{-1})} w(y)\, d\s_d(y)$ for  $x\in\sph$ and $n=1,2,\cdots.$
 This combined with \eqref{5-15:2017}  implies that for $1\leq p<\infty$,
  \begin{equation}\label{5-17:2017}\|P_n(\la x, \cdot\ra) \|_{p, w}\sim \|P_n (\la x, \cdot\ra)\|_{p, w_n}\sim n^{d-1}(n^{-d+1}w_n(x))^{\f 1p}.\end{equation}
 Thus, using \eqref{5-17:2017} and \eqref{5-16:2017} , we deduce that for any $x\in\sph$,
 \begin{align*}
   (n^{-d+1}w_n(x))^{-\f12}\sim \f {\|P_n(\la x, \cdot\ra)\|_{2,w}}{\|P_n(\la x, \cdot\ra)\|_{1,w}} \leq C \sqrt{N}.
 \end{align*}
 It follows that
 $$ N\ge c ( n^{-d+1} w_n(x))^{-1}=\f c {w(B(x, n^{-1}))},\  \   \  \forall x\in\sph.$$
\section*{Acknowledgement}

 We would like to express  sincere gratitude to Professor Ron Peled from the Tel Aviv University for  kindly pointing out to us several   very helpful  references on Chebyshev-type cubature formulas.

\end{document}